\documentclass[10pt,a4paper]{amsart}

\usepackage{float}
\usepackage{subfig}

\usepackage{amsmath}
\usepackage{amssymb}
\usepackage{amsthm}

\usepackage{graphicx}
\usepackage{hyperref}
\usepackage{caption}

\usepackage{fullpage}
\usepackage{mathrsfs}  
\usepackage{tensor}
\usepackage{url}
\usepackage{adjustbox}
\usepackage{arydshln}
\usepackage{mathtools}

\usepackage{fancyvrb}
\usepackage{tikz}
\usetikzlibrary{positioning}
\tikzset{
  node/.style={on grid},
}
\usetikzlibrary{arrows}

\usepackage{autonum}	

\makeatletter
\def\@endtheorem{\endtrivlist}
\makeatother

\newtheorem{thm}{Theorem}[section]
\newtheorem{lem}[thm]{Lemma}

\newtheorem{prop}[thm]{Proposition}

\newtheorem{defn}{Definition}[section]
\theoremstyle{remark}
\newtheorem{rem}{Remark}

\numberwithin{equation}{section}

\begin{document}

\title[\(T^2\)-Symmetric Expanding Stability]{Stability Within \(T^2\)-Symmetric Expanding Spacetimes}

\author{Beverly K. Berger}
\address{Edward L. Ginzton Laboratory\\
Stanford University\\
Stanford, CA 94305-4088\\
USA}
\email{beverlyberger@me.com}

\author{James Isenberg}
\address{Department of Mathematics\\
University of Oregon\\
Eugene, OR  97403-1222\\
USA}
\email{isenberg@uoregon.edu}

\author{Adam Layne}
\address{Department of Mathematics\\
KTH Royal Institute of Technology\\
100 44 Stockholm\\
Sweden}
\email{layne@kth.se}

\date{\today}

\begin{abstract}
We prove a nonpolarised analogue of the asymptotic characterization of \(T^2\)-symmetric Einstein Flow solutions completed recently by LeFloch and Smulevici.
In this work, we impose a condition weaker than polarisation and so our result applies to a larger class.
We obtain similar rates of decay for the normalized energy and associated quantities for this class.
We describe numerical simulations which indicate that there is a locally attractive set for \(T^2\)-symmetric solutions not covered by our main theorem.
This local attractor is distinct from the local attractor in our main theorem, thereby indicating that the polarised asymptotics are unstable.
\end{abstract}

\maketitle

\section{introduction}

There exist broad conjectures about the expanding direction behavior of vacuum spacetimes with closed Cauchy surfaces \cite{MR1888088,MR1894914}, but currently little is known about some of the most elementary examples.
Recent results \cite{MR3513138,MR3312437} have demonstrated that certain vacuum cosmological models demonstrate locally stable behavior in the expanding direction, but that well-known subclasses are unstable.
These results should be compared to models with matter \cite{Radermacher:2017mkr,MR2960029,MR3874696,MR3739229} where spatially homogeneous solutions are known to be stable.
It is also important to note that the behavior of these models in the direction of the singularity is not sensitive to the presence of most types of matter \cite{MR0363384}.

In the special case that the spacetime has spatial topology \(T^3\), admits two spacelike Killing vector fields (such spacetimes are called \emph{\(T^2\)-symmetric}), and satisfies a further technical condition (that the spacetime is \emph{polarised}) results of \cite{MR3513138} show that there is a local attractor of the Einstein Flow in the expanding direction.
It is natural to ask whether the condition that the spacetime be polarised is necessary.
Do spacetimes on \(T^3\) with two spacelike Killing vector fields necessarily become effectively polarised?
Do they then flow to the polarised attractor?

We partially resolve these questions by analytic and numerical means.
Our main theorem states that solutions which are not polarised have the expanding direction asymptotics of polarised solutions if they satisfy a certain weaker condition: that one of the two conserved quantities of the flow vanishes.
We call such solutions \emph{\(B_0\)} or \emph{\(B=0\) solutions}.
The conserved quantity \(B\) vanishes for all polarised solutions; the set of \(B=0\) solutions is of codimension one in the space of all solutions in these coordinates while the set of polarised solutions is of infinite codimension.

It was shown in \cite{MR1474313} that \(T^2\)-symmetric vacuum spacetimes posess a global foliation; all such Einstein Flows have a metric of the form
\begin{align}
		g
=	&	e^{\widehat l - V + 4\tau} \left( - d\tau^2 + e^{2(\rho - \tau)} d\theta^2 \right) 	+e^{V} \left[ dx + Q dy + (G + Q H ) d\theta \right]^2 +  e^{-V+2\tau} \left[ dy + H d\theta \right]^2 				\label{coords}	
\end{align}
where $\partial_x$ and $\partial_y$ are the Killing vector fields.
The area of the \(\{\partial_x, \partial_y\}\) orbit is \(e^{2\tau}\), so the singularity occurs as \(\tau \to -\infty\) and the spacetime expands as \(\tau \to \infty\).
Relative to the coordinates \(t,P,\alpha,\lambda\) used in  \cite{MR3312437}, our quantities are given by
\begin{align}
\begin{array}{rlrl}
   \tau :=&  \log t	,				&
  \rho  :=& -\frac{1}{2} \log \alpha 					\\
V :=& P + \log t 		,			&
    \widehat l:=&P + \frac{1}{2} \lambda - \frac{3}{2} \log t. 					
\end{array}
\end{align}
See the Appendix for a complete concordance of notations between the cited papers and the present work.
In the coordinates \eqref{coords}, the Einstein Flow is
\begin{align}
\partial_\tau \left( e^\rho V_\tau \right)
=& \partial_\theta \left( e^{2\tau - \rho} V_\theta \right)+ e^{2(V-\tau)+\rho} \left( Q_\tau^2 - e^{2(\tau-\rho)} Q_\theta^2 \right) \label{v transport}					\\
\partial_\tau \left( e^{\rho + 2 (V - \tau)} Q_\tau \right) \label{q transport}
=& \partial_\theta \left( e^{-\rho +2 V} Q_\theta \right)				\\
\widehat l_\tau + \rho_\tau + 2 
 =&\frac{1}{2} \left[ V_\tau^2 + e^{2(\tau-\rho)} V_\theta^2 + e^{2(V-\tau)} \left( Q_\tau^2 + e^{2(\tau-\rho)}  Q_\theta^2 \right) \right]\label{l-evol} 					\\
 \rho_\tau
 =&  K^2e^{\widehat l}\label{rho-evol}					\\
 \widehat l_\theta=&V_\theta V_\tau + e^{2(V-\tau)} Q_\theta Q_\tau. \label{constraint}					
\end{align}
The last equation is the momentum constraint, and it is preserved by the evolution equations.
Equation \eqref{rho-evol} is a consequence of the constraints; \(\rho\) satisfies a wave equation similar to \eqref{v transport} which can be derived as a consequence of \eqref{l-evol} and \eqref{rho-evol}, so we take equations \eqref{v transport} through \eqref{rho-evol} to be the evolution equations instead.
There are, in addition, evolution equations for \(G,H\), but these may be integrated once \(V,Q,\rho,\widehat l\) have been found, so these latter four functions are the ones of interest.
As a consequence of \eqref{v transport} and \eqref{q transport}, there are two conserved quantities along the flow:
\begin{align}
A:=&\int_{S^1}  e^\rho \left( V_\tau  - e^{2(V-\tau)} Q_\tau Q \right) \, d\theta			\\
B:=&\int_{S^1}  e^{\rho+2(V-\tau)} Q_\tau  \, d\theta.			
\end{align}
The condition \(Q \equiv 0\) is often imposed when studying these solutions in the collapsing direction.
Such solutions are called \emph{polarised}.
(Note that all polarised solutions have \(B = 0\), but not all \(B = 0\) solutions are polarised.)
The constant \(K\) is, without loss of generality, that ``twist constant'' which cannot in general be made to vanish by a coordinate transformation.
The \(T^3\) \emph{Gowdy models} \cite{MR0339764} are those for which \(K =0\).
\(T^2\)-symmetric spacetimes which are polarised, half polarised \cite{MR1939922,clausenthesis}, or Gowdy have been studied extensively in the contracting direction (e.g. \cite{MR3085923}).
We are here concerned only with the expanding direction.

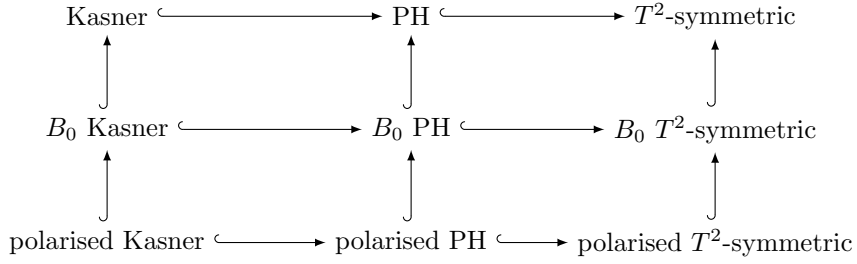
\begin{figure}
   \centering
   \subfloat{
\begin{tikzpicture}[%
  back line/.style={right hook-latex},
  rcross line/.style={preaction={draw=white, -,line width=6pt}, right hook-latex},
  lcross line/.style={preaction={draw=white, -,line width=6pt}, left hook-latex}]

  \node (A) {Kasner};
  \node [right of=A, node distance=4cm] (C) {PH};
  \node [right of=C, node distance=4cm] (D) {\(T^2\)-symmetric};
  \draw[rcross line] (A) -- (C);
  \draw[rcross line] (C) -- (D);

  \node [below of=A, node distance=1.5cm] (A1) {\(B_0\) Kasner};
  \node [below of=C, node distance=1.5cm] (C1) {\(B_0\) PH};
  \node [below of=D, node distance=1.5cm] (D1) {\(B_0\) \(T^2\)-symmetric};
  \draw[rcross line] (A1) -- (C1);
  \draw[rcross line] (C1) -- (D1);

  \node [below of=A1, node distance=1.5cm] (A2) {polarised Kasner};
  \node [below of=C1, node distance=1.5cm] (C2) {polarised PH};
  \node [below of=D1, node distance=1.5cm] (D2) {polarised \(T^2\)-symmetric};
  \draw[rcross line] (A2) -- (C2);
  \draw[rcross line] (C2) -- (D2);
  
  \draw[rcross line] (A2) -- (A1);
  \draw[rcross line] (A1) -- (A);
  \draw[rcross line] (C2) -- (C1);
  \draw[rcross line] (C1) -- (C);
  \draw[rcross line] (D2) -- (D1);
  \draw[rcross line] (D1) -- (D);
 
\end{tikzpicture}
	}
	\caption[Classes of Einstein Flows under consideration.]{The classes of Einstein Flow solutions discussed in this paper, and their inclusions.
	We have omitted the Gowdy models, which are not the focus of this work.}\label{com}
\end{figure}

The \emph{Kasner} models are those which are spatially homogeneous (\(\widehat l,V,Q\) are independent of \(\theta\)) and satisfy \(K= 0\).
Let us note that, in our coordinates, polarised Kasner solutions \cite{MR1501305} take the form
\begin{align}
V = a \tau + b , \quad \widehat l = \left(\frac{1}{2}a^2 - 2 \right) \tau + c
\end{align}
for some constants \(a,b,c \in \mathbb R\).
The Gowdy models contain all Kasners, and in the expanding direction the dynamics of Gowdy solutions are known \cite{MR2032917}, \cite{2017arXiv170702803R} and appear to be very different from those of non-Gowdy solutions.
Non-Gowdy solutions such that \(\widehat l,V,Q\) are independent of \(\theta\) are called \emph{pseudo-homogeneous} or \emph{PH}.
This definition appears in \cite{MR3312437}, where it is shown that the future asymptotics are of the form
\begin{align}
\left|V -( a \tau + b ) \right| \to 0, \quad \left|\widehat l -\left( \left[\frac{1}{2}a^2 - 2 \right] \tau + c\right) \right| \to 0 , \quad a\in (-2,2).
\end{align}
That is, PH solutions have asymptotics of the same form as a Kasner solution, but the value of \(V_\tau\) at \(\tau  = \infty\) is restricted.

In contrast to these examples, in \cite{MR3513138} the authors find a set of non-Gowdy, polarised solutions such that 
\begin{align}
\left|V -b \right| \to 0, \quad \left|\widehat l - c \right| \to 0 . \label{ls asymptotics}
\end{align}
The results in \cite{MR3312437} and \cite{MR3513138} are much more detailed than the above statements; we give this simple description only to demonstrate that an instability arises; no polarised Kasner or PH solutions can have future behavior of the form \eqref{ls asymptotics}.
The relationships between these sets of solutions are given in Figure \ref{com}.

Previous to this work, numerical simulations conducted by Berger \cite{bergerpcgm, bergeraps} indicated that all \(T^2\)-symmetric solutions, without regard to the polarisation or smallness conditions imposed in \cite{MR3513138}, flowed toward the polarised attractor \eqref{ls asymptotics}.
In addition, in \cite{MR3312437} it is shown that within the neighborhood of each polarised PH solution is a polarised non-PH solution with future asymptotics of the form \eqref{ls asymptotics}.

Before giving a description of our main theorem, let us note the sense in which we use the word ``attractor.''
Our technique of proof follows \cite{MR3513138}.
Let us denote the right side of \eqref{l-evol} by \(J\).
The idea of the proof is to treat the asymptotic regime of the solution as a wave equation for \(V,Q\) coupled to an ordinary differential equation (up to some error terms) for the means in the \(\theta\)-direction of \(e^\rho, e^{\widehat l}, J\).
The smallness assumptions are then used to guarantee that the errors decay, and so the behavior of the means of \(  e^\rho , e^{\widehat l}, J\) approaches the behavior of the solution of the ODE.
When we use the word ``attractor'' here, we refer to the dynamics of the \(e^\rho, e^{\widehat l}, J\) system; a solution \( V',  Q',  \rho',  \widehat l'\) is not generally a proper attractor of the flow in the sense that
\begin{align}
\|V -  V'\| + \| Q -  Q'\| + \| \rho -  \rho' \| + \| \widehat l - \widehat l '\| \longrightarrow 0
\end{align}
in \(C^k\) norm.

Our main theorem states roughly that the condition \(B=0\) suffices to ensure that a solution has polarised asymptotics if it begins sufficiently close to the asymptotic regime.
In the latter portion of the paper, we present numerical evidence that the condition \(B=0\) is necessary for the solution to have polarised asymptotics and flow toward the polarised attractor.
There appears to be an attractor for solutions satisfying \(B\ne 0\), which shares some formal properties with the \(B=0\) attractor.
However, such solutions flow away from the \(B=0\) attractor, and so the \(B=0\) asymptotics appear to be unstable.

Since the future behavior of Gowdy and PH solutions is understood, we are only concerned with non-Gowdy, non-PH solutions; that is, solutions with \(K\ne 0\) and \(\int_{S^1} e^\rho \, d\theta\) unbounded as \(\tau \to \infty\).
In this case, we shift \(\widehat l\) by a constant 
\begin{align}
l := \widehat l + \log ( K^2 /2 )			
\end{align}
so that
\begin{align}
\widehat l_\theta = l_\theta, \quad \widehat l_\tau = l_\tau \quad \text{and} \quad \rho_\tau  = e^l.
\end{align}
In the rest of the paper, we assume solutions are non-Gowdy and so change variables to \(l\) to avoid writing factors of \(K\).

Before proceeding with the proof, it is important to note that there is some very interesting work on the rescaling limits of certain expanding spacetimes \cite{2017arXiv170105150L,2018CQGra..35c5010L}.
The earlier of these works uses techniques from the study of Ricci Flow to analyze the rescaling limits of CMC-foliated expanding spacetimes.
The latter work is concerned with the extent to which rescaling limits of the spacetimes considered in \cite{MR3513138} have a nonzero Einstein tensor.
It is likely that this result can be generalized to the class of solutions considered in this paper.

\textit{Acknowledgements.}
We are grateful to David Maxwell, Peng Lu, Paul T Allen, Florian Beyer, Piotr Chru\'sciel, Anna Sakovich, and Hans Ringstr\"om for providing useful comments on various parts of this project.

The second and third authors were supported by NSF grants DMS-1263431 and PHY-1306441.
This article was in part written during a stay of the third author at the Erwin Schr\"odinger Institute in Vienna during the thematic program `Geometry and Relativity'.
This paper incorporates material that previously appeared in the third author's dissertation which was submitted to the Department of Mathematics in partial fulfillment of the requirements for the degree of Doctor of Philosophy at the University of Oregon.

\section{Preliminary Computations}

\label{prelim chapter}
Before proceeding with the proof of the main theorem, we define the energy under consideration and calculate its evolution.
It is useful to have notation for the mean of a function in the \(\theta\)-direction.

\begin{defn}[\(S^1\)-mean]
For \(f \colon S^1 \to \mathbb R\), let
\begin{align}
\langle f \rangle  := \int_{S^1} f(\theta) \, d\theta.
\end{align}
\end{defn}
Note that in \cite{MR3513138}, the authors choose to use the volume form \( e^\rho \, d\theta \) for their mean.
Our choice is almost identical to that used in \cite{MR3312437}, but we normalize so that \(\int_{S^1} d\theta = 1\).
Either choice would suffice.

Define the following energy
\begin{align}
J:=&\frac{1}{2} \left[ V_\tau^2 + e^{2(\tau - \rho )}V_\theta^2 + e^{2(V-\tau)} \left( Q_\tau^2 + e^{2(\tau - \rho )}Q_\theta^2 \right) \right], \\
E:=& \int_{S^1}e^{\rho -2\tau} J \, d\theta =\frac{1}{2}\int_{S^1}   e^{\rho -2\tau}V_\tau^2 + e^{- \rho }V_\theta^2 +  e^{2V-4\tau+\rho}Q_\tau^2 + e^{2(V -\tau)- \rho }Q_\theta^2   \, d\theta,
\end{align}
and the \(S^1\)-volume
\begin{align}
\Pi:=& \langle e^\rho \rangle  =\int_{S^1}e^{\rho}  \, d\theta .
\end{align}
Note that equation \eqref{l-evol} now reads \(l_\tau + \rho_\tau + 2  =J \).
We use the terms \(V\)-energy and \(Q\)-energy loosely to refer to \(V_\tau^2 + e^{2(\tau - \rho )}V_\theta^2\) and \(e^{2(V-\tau)} \left( Q_\tau^2 + e^{2(\tau - \rho )}Q_\theta^2 \right)\), respectively.
One may compute using the evolution equations for $V$ and $Q$ that
\begin{align}
		\partial_\tau \left(e^\rho J \right)
=	&	2e^\rho J - \rho_\tau e^\rho J - e^\rho V_\tau^2 - e^{2V - \rho} Q_\theta^2 + \partial_\theta \left( e^{2\tau - \rho} V_\theta V_\tau + e^{2V - \rho} Q_\theta Q_\tau \right)					
\end{align}
so the energy \(E\) evolves by
\begin{align}
E_\tau = \int_{S^1} -\rho_\tau e^{\rho - 2 \tau} J - e^{\rho-2\tau} V_\tau^2 - e^{2(V-\tau) - \rho} Q_\theta^2\, d\theta.
\end{align}
The terms \(- e^{\rho-2\tau} V_\tau^2 - e^{2(V-\tau) - \rho} Q_\theta^2\) appearing here are undesirable for proving energy inequalities.
This necessitates the modification of \(E\) by a term which trades \(V_\tau^2\) for \(V_\theta^2\).
This is the main topic of Section \ref{correction chapter}.

\section{corrections and their bounds}

\label{correction chapter}
Define the correction
\begin{align}\label{cor def}
\Lambda :=& \frac{1}{2} e^{-2\tau} \int_{S^1} V_\tau  \left( V - \langle V \rangle - 1\right)e^\rho \, d\theta.
\end{align}
Corrections to the energy of essentially this form were used previously in the Gowdy case \cite{MR2032917} and in the existing results on $T^2$-symmetric spacetimes \cite{MR3312437, MR3513138}.
Our definition differs only slightly from those previously used.
Differentiating \eqref{cor def} and using integration by parts yields the two components of the \(V\)-energy, but with opposite sign.
This allows us to replace time derivatives by space derivatives, which may be bounded.
At the same time, the correction has better decay properties than the energy, and so we are able to draw conclusions about the energy in the expanding direction.

To trade \(V_\tau^2\) for \(V_\theta^2\) and \(Q_\tau^2\) for \(Q_\theta^2\), it would be more natural to consider the corrections
\begin{align}
\frac{1}{2} e^{-2\tau} \int_{S^1} V_\tau  \left( V - \langle V \rangle \right)e^\rho \, d\theta,\quad \text{and}\quad 
\frac{1}{2} e^{-2\tau} \int_{S^1} e^{2(V-\tau)} Q_\tau  \left( Q - \langle Q \rangle \right)e^\rho \, d\theta
\end{align}
separately as other authors have done.
Then, by differentiating the \(Q\)-correction one would hope to obtain terms of the form \(Q_\tau^2 - e^{2(\tau - \rho )}Q_\theta^2 \), perhaps with a leading factor.
Our definition exploits the fact that \eqref{v transport} contains exactly the expression that we would like to obtain from the \(Q\)-correction.

\begin{lem}
Consider a non-Gowdy \(T^2\)-symmetric Einstein Flow.
The correction defined in \eqref{cor def} evolves by
\begin{align}
\partial_\tau \Lambda\ =& -2 \Lambda - \frac{1}{2}  \left \langle   e^{-\rho} V_\theta^2 \right \rangle + \frac{1}{2}  \left \langle e^{\rho-2\tau}V_\tau^2 \right \rangle-\frac{1}{2} \left \langle V_\tau \right \rangle \left \langle  e^{\rho-2\tau} V_\tau   \right \rangle\\
&+ \frac{1}{2} e^{-2\tau} \int_{S^1} e^{2(V - \tau) + \rho} \left ( Q_\tau^2 - e^{2(\tau - \rho)} Q_\theta^2 \right)  \left( V - \left \langle V \right \rangle -1\right)\, d\theta .
\end{align}
\end{lem}

\begin{proof}
We compute straightforwardly using equations \eqref{v transport}, \eqref{q transport} and integration by parts.
From the definition of \(\Lambda\) we have
\begin{align}
\partial_\tau \Lambda 
=& -2 \Lambda + \frac{1}{2} e^{-2\tau} \int_{S^1}(e^\rho  V_\tau)_\tau   \left( V - \left \langle V \right \rangle -1\right)\, d\theta 
+ \frac{1}{2} e^{-2\tau} \int_{S^1} V_\tau  \partial_\tau\left( V - \left \langle V \right \rangle -1\right)e^\rho \, d\theta\\
=& -2 \Lambda \\
&+ \frac{1}{2} e^{-2\tau} \int_{S^1}\left[ e^{2\tau} \left( e^{-\rho} V_\theta \right)_\theta + e^{2(V - \tau) + \rho} \left ( Q_\tau^2 - e^{2(\tau - \rho)} Q_\theta^2 \right) \right]   \left( V - \left \langle V \right \rangle -1\right)\, d\theta \\
&+ \frac{1}{2} e^{-2\tau} \int_{S^1} V_\tau  \partial_\tau\left( V - \left \langle V \right \rangle-1 \right)e^\rho \, d\theta\\
=& -2 \Lambda + \frac{1}{2} e^{-2\tau} \int_{S^1} -e^{2\tau}  e^{-\rho} V_\theta^2  \, d\theta + \frac{1}{2} e^{-2\tau} \int_{S^1} V_\tau^2e^\rho \, d\theta\\
&+ \frac{1}{2} e^{-2\tau} \int_{S^1} e^{2(V - \tau) + \rho} \left ( Q_\tau^2 - e^{2(\tau - \rho)} Q_\theta^2 \right)  \left( V - \left \langle V \right \rangle -1\right)\, d\theta \\
&- \left \langle V_\tau \right \rangle \left \langle  \frac{1}{2}e^{\rho-2\tau} V_\tau   \right \rangle
\end{align}
which completes the proof.
\qedhere
\end{proof}

We modify the energy \(E\) by \(\Lambda\).
It is then desirable to know that \(\Lambda\) has better decay than \(E\).
To that end, note that
\begin{align}
\| V - \langle V \rangle \|_{C^0} 
\lesssim		\int_{S^1} | V_\theta | \, d\theta 
\leq			\left( \int_{S^1}  V_\theta^2 e^{-\rho}  \, d\theta \right)^{1/2} \Pi^{1/2}
\leq			(\Pi E)^{1/2}.\label{v sup bound}
\end{align}
As is standard (cf. \cite{2014arXiv1407.6298R}), we use the notation \(f \lesssim h\) to mean that there is a universal constant \(C > 0 \) such that \(f \leq C h\).

One finds the following bound using H\"older's Inequality.
\begin{lem}[\cite{MR3312437}, Lemma 72]
Consider a non-Gowdy \(T^2\)-symmetric Einstein Flow.
Then
\begin{align}\label{correction bound 1}
				\left| \Lambda + \left\langle \frac{1}{2} e^{\rho-2\tau}V_\tau \right\rangle \right| 
=				\left| \frac{1}{2} e^{-2\tau} \int_{S^1} V_\tau  \left( V - \langle V \rangle\right)e^\rho \, d\theta \right| 
\lesssim	e^{-\tau} \Pi E 
\end{align}
\end{lem}

For the following bound on the \(Q\) correction, cf. \cite{MR3312437} Lemma 73, where the author assumes a uniform bound on \(\Pi\) which we don't assume here.
The proof is essentially the same.
\begin{lem}
For any a non-Gowdy \(T^2\)-symmetric Einstein Flow,
\begin{align}
\left|  e^{-2\tau} \int_{S^1} e^{2(V-\tau)} Q_\tau  \left( Q - \langle Q \rangle \right)e^\rho \, d\theta \right| \lesssim 	e^{-\tau}e^{2(\Pi E)^{1/2}}  \Pi E 
\end{align}
\end{lem}
\begin{proof}
Note that we have already bounded \(\| V - \langle V \rangle \|_{C_0}\) in equation \eqref{v sup bound}, and so we may commute out factors of \(e^V\) to obtain
\begin{align}
					\left \| e^{V}  \left( Q - \langle Q \rangle \right) \right \|_{C^0}
=				&	\left \| e^{V - \langle V \rangle +\langle V \rangle }  \left( Q - \langle Q \rangle \right) \right \|_{C^0}\\
=				&	e^{\|V - \langle V \rangle\|_{C^0}} e^{\langle V \rangle } \left \|    Q - \langle Q \rangle  \right \|_{C^0}\\
\leq			&	e^{2\|V - \langle V \rangle\|_{C^0}}  \left( \int_{S^1}  e^{2V}Q_\theta^2 e^{-\rho}  \, d\theta \right)^{1/2} \Pi^{1/2}\\
\leq			&	e^{2\|V - \langle V \rangle\|_{C^0}}  e^\tau \left(E \Pi\right)^{1/2}
\end{align}
via H\"older's inequality.
So we may compute, using the bound on \(\| V - \langle V \rangle \|_{C_0}\), H\"older's inequality, and the definition of \(E\)
\begin{align}
					\left |  e^{-2\tau} \int_{S^1} e^{2(V-\tau)} Q_\tau  \left( Q - \langle Q \rangle \right)e^\rho \, d\theta \right | 
\lesssim	&	e^{-4\tau} \left \| e^{V}  \left( Q - \langle Q \rangle \right) \right \|_{C^0} \left | \int_{S^1} e^{V} Q_\tau e^\rho \, d\theta \right |\\
\lesssim	&	e^{2\|V - \langle V \rangle\|_{C^0}}  e^{-3\tau} E^{1/2} \Pi^{1/2}\left | \int_{S^1} e^{V} Q_\tau e^\rho \, d\theta \right |\\
\leq			&	e^{2\|V - \langle V \rangle\|_{C^0}}  e^{-\tau} E \Pi\\
\lesssim	&	e^{-\tau}e^{2(\Pi E)^{1/2}}  \Pi E .
\qedhere
\end{align}
\end{proof}

We only need the \(Q\) correction for the following identity, which follows directly from the definitions of the conserved quantities \(A,B\):
\begin{align}
	\left \langle  e^{\rho} V_\tau   \right \rangle 
=	A+ B\langle Q \rangle +  \int_{S^1} e^{2(V-\tau)} Q_\tau  \left( Q - \langle Q \rangle \right)e^\rho \, d\theta.
\end{align}
For \(B_0\) solutions, however, we use the bound on the \(Q\) correction to obtain the following bound
\begin{align}
\left| \left \langle  e^{\rho-2\tau} V_\tau   \right \rangle \right| -e^{-2\tau}|A| \lesssim    e^{-\tau}e^{2(\Pi E)^{1/2}}  \Pi E  \label{correction bound 2}
\end{align}
which together with \eqref{correction bound 1} yields the desired estimate on the correction.
\begin{prop}
For any a non-Gowdy, \(B_0\) \(T^2\)-symmetric Einstein Flow,
\begin{align}
				\left| \Lambda  \right| -\frac{e^{-2\tau}}{2}|A| 
\lesssim  e^{-\tau}	\left( 1 +   e^{2(\Pi E)^{1/2}}   \right)\Pi E.\label{correction bound 3}
\end{align}
\end{prop}

The correction \(\Lambda\) introduces significant new error terms after differentiation.
However, these terms have good bounds, and the modified energy \(E + \Lambda\) has significantly better properties upon comparison to \(E\) alone.
The evolution of this modified energy is the focus of the next chapter.

\section{The corrected energy}

\label{corrected chapter}
One would like to show that, up to error terms, \(\Pi\) and \(E\) satisfy an ODE.
While this is true asymptotically, it is more useful to compute with an energy which has been modified by the correction.

One computes that
\begin{align}
 \left( E + \Lambda \right)_\tau 
=&   
\int_{S^1}   -e^{\rho-2\tau}  \rho_\tau J -e^{\rho-2\tau} V_\tau^2- e^{2(V - \tau) - \rho }Q_\theta^2    \, d\theta
-2 \Lambda \\
&+ \frac{1}{2} e^{-2\tau} \int_{S^1} 
-e^{2\tau}  e^{-\rho} V_\theta^2  \, d\theta 
+ \frac{1}{2} e^{-2\tau} \int_{S^1} V_\tau^2e^\rho \, d\theta 
 \\
&+ \frac{1}{2} e^{-2\tau} \int_{S^1} e^{2(V - \tau) + \rho} \left ( Q_\tau^2 - e^{2(\tau - \rho)} Q_\theta^2 \right)  \left( V - \left \langle V \right \rangle -1\right)\, d\theta \\
&-\left \langle V_\tau \right \rangle \left \langle  \frac{1}{2}e^{\rho-2\tau} V_\tau   \right \rangle  \\
=&   -\left(1 +  \frac{\Pi_\tau}{\Pi}\right) \left( E + \Lambda \right)+\left( \frac{\Pi_\tau}{\Pi}E -\int_{S^1}   e^{\rho-2\tau}  \rho_\tau J \, d\theta   \right) -\left(1 -  \frac{\Pi_\tau}{\Pi}\right) \Lambda \\
&+ \frac{1}{2} e^{-2\tau} \int_{S^1} e^{2(V - \tau) + \rho} \left ( Q_\tau^2 - e^{2(\tau - \rho)} Q_\theta^2 \right)  \left( V - \left \langle V \right \rangle \right)\, d\theta  -\left \langle V_\tau \right \rangle \left \langle  \frac{1}{2}e^{\rho-2\tau} V_\tau   \right \rangle.
\end{align}
The leading term on the right leads us to the ansatz that \(\Pi\left(E+\Lambda\right)\) (and so \(\Pi E\)) should decay like \(e^{-\tau}\).
Accordingly, define the corrected, normalized energy \(H:= \Pi\left(E+\Lambda\right) \).
One computes that
\begin{align}
 \partial_\tau \left( e^\tau H \right) 
=& e^\tau H+e^\tau  \Pi_\tau \left( E + \Lambda \right) +e^\tau  \Pi \left( E + \Lambda \right)_\tau  \\
=&  e^\tau \Pi\left( \left( E + \Lambda \right) \left(1 +  \frac{\Pi_\tau}{\Pi}\right) +\left( E + \Lambda \right)_\tau  \right) \\
=& e^\tau \Pi\left[\left( \frac{\Pi_\tau}{\Pi}E -\int_{S^1}   e^{\rho-2\tau}  \rho_\tau J \, d\theta   \right)
-\left(1 -  \frac{\Pi_\tau}{\Pi}\right) \Lambda \right.\label{e tau h evol}\\
&\left.+ \frac{1}{2} e^{-2\tau} \int_{S^1} e^{2(V - \tau) + \rho} \left ( Q_\tau^2 - e^{2(\tau - \rho)} Q_\theta^2 \right)  \left( V - \left \langle V \right \rangle \right)\, d\theta -\left \langle V_\tau \right \rangle \left \langle  \frac{1}{2}e^{\rho-2\tau} V_\tau   \right \rangle \right] 
\end{align}

The ansatz in the local stability proof is that \(e^\tau H \) is of constant order.
The proof is via a bootstrap argument, where we bound all of the terms of \(\partial_\tau (e^\tau H) \) in terms of \(\Pi,E,H\) and \(\tau\).
The following Proposition deals with each of these error terms.

\begin{prop}\label{main estimates}
Consider the evolution of a \(B_0\) solution with initial data given at time \(\tau  = s_0\).
Let \(\rho_0 := \min_{\theta \in S^1} \rho(\theta, s_0) \).
The following estimates hold.
\begin{align}
\left|
\frac{\Pi_\tau}{\Pi} E - \int_{S^1}   e^{\rho-2\tau}  \rho_\tau J     \, d\theta 
\right|
\lesssim& E \int_{S^1} e^{\rho - \tau} \rho_\tau J \, d\theta,\label{j bound}\\
\left|
\left(1 -  \frac{\Pi_\tau}{\Pi}\right)\Lambda
\right|
\lesssim& |A|e^{-2\tau}\left(1 +  \frac{\Pi_\tau}{\Pi}\right) +e^{-\tau} \left(1+e^{2(\Pi E)^{1/2}}\right)\left( \Pi + \Pi_\tau \right) E,\label{lambda bound}\\
\left|\left \langle V_\tau \right \rangle \left \langle  \frac{1}{2}e^{\rho-2\tau} V_\tau   \right \rangle \right|
\lesssim& e^{-\rho_0/2} e^{-\tau} \left( |A| + e^{\tau}e^{2(\Pi E)^{1/2}}  \Pi E \right)E^{1/2},	\label{bad v}
\end{align}
and
\begin{align}
\left| e^{-2\tau}\int_{S^1} e^{2(V - \tau) + \rho} \left ( Q_\tau^2 - e^{2(\tau - \rho)} Q_\theta^2 \right)  \left( V - \left \langle V \right \rangle \right)\, d\theta \right|
\lesssim& \Pi^{1/2} E^{3/2} .\label{strange bound}
\end{align}
\end{prop}

\begin{proof}
For \eqref{j bound}, using Young's inequality, we note that
\begin{align}
|l_\theta| 
\leq	&	|V_\tau V_\theta | + | e^{V-\tau} Q_\tau e^{V-\tau } Q_\theta | \\
=		&	|e^{(\rho-\tau)/2}V_\tau e^{-(\rho-\tau)/2}V_\theta | + | e^{V-\tau} e^{(\rho-\tau)/2}Q_\tau e^{V-\tau } e^{-(\rho-\tau)/2}Q_\theta | \\
\leq	&	\frac{1}{2} \left[e^{\rho-\tau}V_\tau^2 + e^{\tau - \rho}V_\theta^2 +  e^{2(V-\tau)} e^{\rho-\tau}Q_\tau^2 +  e^{2(V-\tau) } e^{\tau - \rho}Q_\theta^2 \right]\\
=		&	e^{\rho - \tau} J. \label{l theta bound}
\end{align}
Thus we may use the Poincar\'e inequality to compute that
\begin{align}
						\left| \frac{\Pi_\tau}{\Pi}E - \int_{S^1}   e^{\rho-2\tau}  \rho_\tau J     \, d\theta \right| 
=					&	\Pi^{-1} \left| \Pi_\tau E - \Pi \int_{S^1}   e^{\rho-2\tau}  \rho_\tau J     \, d\theta \right|  \\
=					&	\Pi^{-1} \left| \int_{S^1} \int_{S^1} e^{\rho(\phi)} e^{\rho(\theta) - 2\tau} J(\theta) \left( \rho_\tau (\phi) - \rho_\tau(\theta) \right) \, d\phi d\theta\right|  \\
\leq				&	\Pi^{-1} \left| \int_{S^1} \int_{S^1} e^{\rho(\phi)} e^{\rho(\theta) - 2\tau} J(\theta) \sup_{a,b \in S^1} | \rho_\tau(a) - \rho_\tau (b) | \, d\phi d\theta\right|  \\
=					&	\Pi^{-1} \Pi   \int_{S^1}  e^{\rho(\theta) - 2\tau} J(\theta)  \,  d\theta \sup_{a,b \in S^1} | \rho_\tau(a) - \rho_\tau (b) | \\
\lesssim		&	E \int_{S^1} \rho_\tau |l_\theta| \,d\theta \\
\leq				&	E \int_{S^1} \rho_\tau  e^{\rho-\tau} J \,d\theta.
\end{align}

Inequality \eqref{lambda bound} follows directly from inequality \eqref{correction bound 3}.
To prove \eqref{strange bound}, we first commute out the $V$-mean.
\begin{align}
				&	\left| e^{-2\tau}\int_{S^1} e^{2(V - \tau) + \rho} \left ( Q_\tau^2 - e^{2(\tau - \rho)} Q_\theta^2 \right)  \left( V - \left \langle V \right \rangle \right)\, d\theta \right|\\
\leq			&	e^{-2\tau}\|V - \langle V \rangle \|_{C^0} \int_{S^1} e^{2(V - \tau) + \rho} \left | Q_\tau^2 - e^{2(\tau - \rho)} Q_\theta^2 \right|\, d\theta\\
\lesssim			&	e^{-2\tau}( \Pi E )^{1/2} \int_{S^1} e^\rho e^{2(V - \tau) } \left( Q_\tau^2 + e^{2(\tau - \rho)} Q_\theta^2 \right)\, d\theta\\
\lesssim	&( \Pi E )^{1/2} \int_{S^1} e^{\rho - 2\tau} J\, d\theta\\
=				&\Pi^{1/2} E^{3/2}.
\end{align}

Lastly, for \eqref{bad v} recall that \(\rho\) is increasing and compute that
\begin{align}
					\left|  \left \langle V_\tau \right \rangle  \right|
\leq				\left( \int_{S^1} V_\tau^2 e^\rho \, d\theta \right)^{1/2}\left( \int_{S^1} e^{-\rho} \, d\theta \right)^{1/2} 
\leq				e^{-\rho_0/2} e^\tau \left(\int_{S^1} V_\tau^2 e^{\rho-2\tau} \, d\theta\right)^{1/2} 
\lesssim		e^{-\rho_0/2} e^\tau E^{1/2}
\end{align}
and use \eqref{correction bound 2}.
This completes the proof.
\end{proof}

Now that we have an energy satisfying a good differential equation with good bounds on the error, we proceed to the linearization.

\section{Linearization}

\label{ode chapter}

In \cite{MR3513138}, the authors present an argument that certain asymptotic rates of \(\Pi, E\) should be preferred, based on the assumption that \(e^\tau H\) should be of constant order.
In this section we briefly summarize that argument as it appears in our context.
\begin{defn}
Let \(Y := \left \langle  e^{l + \rho + 2\tau} \right \rangle\).
\end{defn}
This quantity has been previously considered; see \cite{MR1858721} where (modulo factors of \(e^\tau\)) it is called the ``twist potential.''

Note that we have defined \(Y\) so that \(Y_\tau = \langle e^{l + \rho + 2\tau}(l_\tau + \rho_\tau +2) \rangle = \langle e^{l + \rho + 2\tau} J \rangle\).
We want to form a system of ordinary differential equations from the means, however.
So we distribute the integral over the product, introducing the error term \(\Omega\).
One computes
\begin{align}
\Pi_\tau =& e^{-2\tau} Y \label{pi evol}\\
Y_\tau  =& e^{2\tau} E Y \Pi^{-1} + \Omega \label{y evol}
\end{align}
where
\begin{align}
\Omega:= \left \langle e^{l+ \rho + 2\tau}  J  \right \rangle - e^{2\tau} E Y \Pi^{-1}
\end{align}
is an error term satisfying
\begin{align}
\left | \Omega \right | \leq e^{4\tau} E \left \langle e^{l+ \rho - \tau}  J  \right \rangle = e^\tau E Y_\tau.\label{ first omega bound}
\end{align}
Note that our quantity $E$ contains the terms $Q_\theta$ and $Q_\tau$, and so is not identical to the energy in \cite{MR3513138}.
Nonetheless, the quantities $\Pi ,Y,E$ satisfy similar relations to the relations that LeFloch and Smulevici's quantities do.
Normalizing, we compute that
\begin{align}
		\partial_\tau \left( e^{-\tau} H^{-1/2} \Pi\right)
=	&	e^{-\tau} H^{-1/2} \Pi_\tau - e^{-\tau} H^{-1/2} \Pi -\frac{1}{2}e^{-\tau} H^{-1/2} \Pi \frac{H_\tau}{H}\\
=	&	\left(e^{-3\tau} H^{-1/2}  Y\right) + \left( e^{-\tau} H^{-1/2} \Pi\right) \left( -1 -\frac{1}{2} \frac{H_\tau}{H}\right) \\
		\partial_\tau \left( e^{-3\tau} H^{-1/2}  Y\right)
=	&	e^{-3\tau} H^{-1/2}  Y_\tau - 3e^{-3\tau} H^{-1/2} Y -\frac{1}{2}e^{-3\tau} H^{-1/2} Y \frac{H_\tau}{H}\\
=	&	e^{-3\tau} H^{-1/2}  \left(e^{2\tau} E Y \Pi^{-1} + \Omega\right) + \left( e^{-3\tau} H^{-1/2} Y\right) \left( -3 -\frac{1}{2} \frac{H_\tau}{H}\right) \\
=	&	 \frac{\left(e^{-3\tau} H^{-1/2}Y\right)}{\Pi^2}e^{2\tau} \Pi E    + \left( e^{-3\tau} H^{-1/2} Y\right) \left( -3 -\frac{1}{2} \frac{H_\tau}{H}\right) +e^{-3\tau} H^{-1/2}\Omega \\
=	&	 \frac{\left(e^{-3\tau} H^{-1/2}Y\right)}{\left( e^{-\tau} H^{-1/2}\Pi \right)^2} \frac{\Pi E}{H}    + \left( e^{-3\tau} H^{-1/2} Y\right) \left( -3 -\frac{1}{2} \frac{H_\tau}{H}\right) +e^{-3\tau} H^{-1/2}\Omega \\
=	&	   \frac{\left(e^{-3\tau} H^{-1/2}Y\right)}{\left( e^{-\tau} H^{-1/2}\Pi \right)^2}   + \left( e^{-3\tau} H^{-1/2} Y\right) \left( -3 -\frac{1}{2} \frac{H_\tau}{H}\right) +e^{-3\tau} H^{-1/2}\Omega \\
	&	+\frac{\left(e^{-3\tau} H^{-1/2}Y\right)}{\left( e^{-\tau} H^{-1/2}\Pi \right)^2} \left(\frac{\Pi E}{H} -1\right).
\end{align}
We insert our ans\"atze that \(\frac{H_\tau}{H} \to -1\), \(e^{-3\tau} H^{-1/2}\Omega \to 0\), and \( \left(\frac{\Pi E}{H} -1\right) \to 0\), to obtain the ODE
\begin{align}
		\partial_\tau {\overline c}
=	&	{\overline d} + {\overline c} \left(  -\frac{1}{2} \right) \\
		\partial_\tau {\overline d}
=	&	   \frac{{\overline d}}{{\overline c}^2}   + {\overline d} \left(  -\frac{5}{2} \right) 
\end{align}
which has a fixed point at
\begin{align}
{\overline c} = \frac{2}{\sqrt{10}} , \quad {\overline d} = \frac{1}{\sqrt{10}}.
\end{align}
So we conjecture that the quantities
\begin{align}
c:=	\frac{\Pi}{e^\tau \sqrt{H}} - \frac{2}{\sqrt{10}}, \quad 
d:=	\frac{Y}{e^{3\tau} \sqrt{H}} - \frac{1}{\sqrt{10}}
\end{align}
decay and compute the evolution of these quantities using \eqref{pi evol} and \eqref{y evol}.
We find that
\begin{align}
\partial_\tau \left( \begin{array}{l} c \\ d \end{array} \right) 
=&
\left( \begin{array}{ll} -1/2 & 1\\ -5/2 & 0 \end{array} \right)\left( \begin{array}{l} c \\ d \end{array} \right) 
-\frac{1}{2}\partial_\tau \log\left( e^\tau H\right)\left( \begin{array}{l} c \\ d \end{array} \right) \\
&-\frac{1}{2} \partial_\tau \log\left( e^\tau H\right)\left( \begin{array}{l}\frac{2}{\sqrt {10}} \\ \frac{1}{\sqrt {10}} \end{array} \right)
+\left( \begin{array}{l} 0 \\ \frac{f(d,c)}{\left(c+\frac{2}{\sqrt{10}} \right)^2 }  +\left( \displaystyle\frac{\Pi E}{H}-1\right)\frac{d+\frac{1}{\sqrt{10}}}{\left(c+ \frac{2}{\sqrt{10}}\right)^2}+ \displaystyle\frac{\Omega}{e^{3\tau} H^{1/2}} \end{array} \right) 
\end{align}
where \(f(c,d)=\frac{10 c-10 d+3 \sqrt{10}}{4} c^2  - \sqrt{10} cd\) has vanishing linear part.
Let
\begin{align}
\widetilde \Omega:= -\frac{1}{2} \partial_\tau \log\left( e^\tau H\right)\left( \begin{array}{l}\frac{2}{\sqrt {10}} \\ \frac{1}{\sqrt {10}} \end{array} \right)
+\left( \begin{array}{l} 0 \\ \frac{f(d,c)}{\left(c+\frac{2}{\sqrt{10}} \right)^2 }  +\left( \displaystyle\frac{\Pi E}{H}-1\right)\frac{d+\frac{1}{\sqrt{10}}}{\left(c+ \frac{2}{\sqrt{10}}\right)^2}+ \displaystyle\frac{\Omega}{e^{3\tau} H^{1/2}} \end{array} \right)
\end{align}
denote the error term of this approximation.
In the end, the following estimate is obtained (cf. \cite{MR3513138}, Proposition 5.1).
\begin{prop}\label{cd bound}
Consider the evolution of a \(B_0\) \(T^2\)-symmetric solution.
Provided the corrected energy $H$ is positive, one has for \(s\geq s_0\)
\begin{align}
\left|\left( \begin{array}{l} c \\ d \end{array} \right)\right|(s) \lesssim & e^{(s_0-s)/4}\left( \frac{ e^{s_0}H(s_0) }{e^s H(s)}\right)^{1/2}\left|\left( \begin{array}{l} c \\ d \end{array} \right)\right|(s_0) +\int^s_{s_0}     e^{(\tau-s)/4}\left( \frac{e^\tau H(\tau) }{e^s H(s)}\right)^{1/2}|\omega(\tau)| \, d\tau,
\end{align}
where 
\begin{align}
|\omega| \lesssim \left| \widetilde \Omega \right|.
\end{align}
\end{prop}

Quickly note a bound on one of the terms appearing in \(\widetilde \Omega\).
\begin{lem}
Consider the evolution of a \(B_0\) \(T^2\)-symmetric solution.
The following estimate holds.
\begin{align}
\left|e^{-3\tau} H^{-1/2} \Omega
\right|
\lesssim& e^{-2\tau} |H|^{-1/2} E Y_\tau.
\label{Omega bound}
\end{align}
\end{lem}
The proof of this lemma proceeds in the same manner as the proof of inequality \eqref{j bound}.
The remaining three terms in \(\widetilde \Omega\) are estimated directly.
In the next section, we perform a bootstrap argument to bound these errors, provided the initial data is sufficiently close to the asymptotic behavior.

\section{The Bootstrap}

\label{bootstrap chapter}

The technique of proof follows \cite{MR3513138}.
The idea is to impose some smallness assumptions on the means of the energy, the \(S^1\) volume, and their derivatives.
We then use a bootstrap argument to show that these assumptions are improved.
The reason for obtaining the estimates of Lemma \ref{main estimates} is to bound the evolution of the corrected energy \(H\).
Let us discuss how that proof goes.
We have computed \(\partial_\tau \left( e^\tau H \right) \) in equation \eqref{e tau h evol}.
Note that we may bound the right side of that equation by an expression of the form
\begin{align}
\left| \partial_\tau \left( e^\tau H \right)  \right|
\lesssim& e^\tau \Pi E F + \widetilde F
= e^\tau H \frac{\Pi E}{H} F+ \widetilde F
\end{align}
where, using the results of Lemma \ref{main estimates} we can write
\begin{align}
F 
:
=&  \int_{S^1} e^{\rho - \tau} \rho_\tau J \, d\theta +e^{-\tau} \left(1+e^{2(\Pi E)^{1/2}}\right)\left( \Pi + \Pi_\tau \right)  + (\Pi E)^{1/2} + e^{-\rho(s_0)/2}e^{2(\Pi E)^{1/2}}  \Pi  E^{1/2}  \label{f def}
\end{align}
and
\begin{align}
\widetilde F 
:
=  |A|\Pi\left(e^{-\tau}\left(1 +  \frac{\Pi_\tau}{\Pi}\right)  + e^{-\rho_0/2}  E^{1/2} \right). \label{f tilde def}
\end{align}
Note that \(F\) and \(\widetilde F\) are nonnegative.
We are then concerned with the quantities
\begin{align}
\int_{s_0}^\infty F(\tau) \, d\tau, \quad \text{and} \quad \int_{s_0}^\infty \widetilde F(\tau) \, d\tau
\end{align}
which bound the evolution of \(e^\tau H\) in the bootstrap proof.

First, however, we need the following version of Gr\"onwall's Lemma, the proof of which is straightforward.

\begin{lem}[Gr\"onwall's Inequality]\label{gronwall}
Let \(\alpha,\beta,f\) be nonnegative smooth functions on the interval \([s_0,s]\).
Suppose \(f\) satisfies the differential inequality
\begin{align}
|f'| \leq \alpha + \beta f.\label{assum} 
\end{align}
Then  
\begin{align}
\left|f(s)-f(s_0) \right| \leq -f(s_0)+\left(f(s_0) +\int_{s_0}^s \alpha(t) \, dt \right)\exp\left( \int_{s_0}^s \beta(t) \, dt \right). 
\end{align}
\end{lem}

\begin{lem}\label{small lemma}
There exist constants $\epsilon,C_1 >0$, \(M > 1\), a time \(s_0 > 0\) depending on \(\epsilon\), and an open set of \(B_0\) Einstein Flows satisfying the following conditions at time \(\tau = s_0\):
\begin{align}
|A| <& 1 \label{a bound}\\
\rho_0:= \inf_{S^1} \rho >& 0 \label{rho0 bound}\\
|c| < & \epsilon \\
|d| < & \epsilon \\
\left | \frac{\Pi E}{H}  -1 \right| < & 1 ,\\
\frac{1}{2} \epsilon^{-1} <e^{s_0} <&2 \epsilon^{-1}  \\
&e^{s_0} H(s_0)+C_1 \epsilon^{1/2} <  M\epsilon e^{s_0} \label{eng upper bound}\\
0<\frac{1}{M} < &e^{s_0} H(s_0)-C_1 \epsilon^{1/2} . \label{eng lower bound}
\end{align}
Furthermore, for all \(\tau \in [s_0 , \infty)\), the following weaker estimates hold:
\begin{align}
|c| < & \epsilon^{1/4} \label{boot first}\\
|d| < & \epsilon^{1/4} \\
\left | \frac{\Pi E}{H}  -1 \right| < & 3\label{boot correction} \\
\frac{1}{2}\left(e^{s_0} H(s_0)-C_1 \epsilon^{1/2} \right)  < e^{\tau} H(\tau) < & 2\left(e^{s_0} H(s_0)+C_1 \epsilon^{1/2} \right)  \label{boot last}
\end{align}
\end{lem}

\begin{rem}
Assumptions \eqref{a bound} and \eqref{rho0 bound} are not strictly needed.
One could omit these assumptions and instead gain terms involving \(A, \rho_0\) in inequalities \eqref{eng upper bound}, \eqref{eng lower bound}, and \eqref{boot last}.
We have added these assumptions just to simplify the notation.
\end{rem}

The technique of proof is a straightforward ``open closed'' argument:
\begin{enumerate}
\item
Suppose estimates \eqref{boot first} to \eqref{boot last} are satisfied for \(\tau \in [s_0,s)\).
\item
We improve each of the five estimates \eqref{boot first} to \eqref{boot last} at \(\tau = s\) by choosing \(\epsilon\) small.
\end{enumerate}

\begin{proof}
\begin{description}
\item[Initial Estimates]
From assumptions \eqref{boot first} to \eqref{boot last}, we have that
\begin{align} 
e^{-\tau} \lesssim H \lesssim \epsilon e^{s_0-\tau}, 
\end{align}
and
\begin{align} 
 \left| \frac{\Pi}{e^\tau \sqrt{H}} - \frac{2}{\sqrt{10}} \right|  = \left| c \right| < \epsilon^{1/4}, \quad \left| \frac{Y}{e^{3\tau} \sqrt{H}} - \frac{1}{\sqrt{10}} \right| = \left| d \right| < \epsilon^{1/4}
\end{align}
so
\begin{align}
				\Pi
\lesssim		&	\left( \frac{2}{\sqrt{10} } + \epsilon^{1/4} \right)  e^{\tau} H^{1/2}
\leq			\left( \frac{2}{\sqrt{10} } + \epsilon^{1/4} \right) \epsilon^{1/2} e^{(s_0+\tau)/2}
\lesssim		\epsilon^{1/2} e^{s_0/2}e^{\tau/2}, \label{pi estimate}\\
				e^{2\tau}\Pi_\tau
=				Y
\lesssim		&	\left( \frac{1}{\sqrt{10} } + \epsilon^{1/4} \right) e^{3\tau} H^{1/2}
\lesssim		\left( \frac{1}{\sqrt{10} } + \epsilon^{1/4} \right) \epsilon^{1/2} e^{(5\tau+s_0)/2}
\lesssim		\epsilon^{1/2} e^{s_0/2}e^{5\tau/2}.\label{y estimate}
\end{align}
Note that \eqref{boot correction} implies that $\Pi E \lesssim  H $ on this interval, which implies that
\begin{align}
			\Pi E 
\lesssim	\epsilon e^{s_0-\tau}
<	\epsilon
, \quad \text{and}\quad
			1+e^{2(\Pi E)^{1/2}} 
\lesssim	1+e^{\epsilon^{1/2}} \lesssim 1
\end{align}
for sufficiently small \(\epsilon\).
The bound on \(\Pi\) and the fact that \(\Pi,Y > 0\) together imply that, for \(a<0\),
\begin{align}
				\int_{s_0}^s e^{(a-1/2)\tau} \Pi_\tau \, d\tau 
\lesssim	& \epsilon^{1/2}	e^{s_0/2}\int_{s_0}^s e^{a\tau} \, d\tau 
\leq	C(a)		\epsilon^{1/2}	e^{(a+1/2)s_0}  
\end{align}
and similarly
\begin{align}
				\int_{s_0}^s e^{(a-5/2)\tau} Y_\tau \, d\tau 
\lesssim	& e^{(a-5/2)s} Y(s) -  e^{(a-5/2)s_0} Y(s_0)  	-(a-5/2)\int_{s_0}^s e^{(a-5/2)\tau} Y \, d\tau \\
\lesssim	& \epsilon^{1/2}\left[e^{as+s_0/2}  	-(a-5/2)e^{s_0/2}\int_{s_0}^s e^{a\tau} \, d\tau \right]\\
\lesssim	& \epsilon^{1/2}\left[e^{as+s_0/2}  	-\frac{a-5/2}{a}\left( e^{as+s_0/2} - e^{(a+1/2)s_0} \right) \right]\\
\lesssim	& \epsilon^{1/2}\left[e^{as+s_0/2}  	+C(a) e^{(a+1/2)s_0}  \right]\\
\lesssim	&	 C(a) \epsilon^{1/2} e^{\left(a+\frac{1}{2}\right)s_0} . 
\end{align}

\item[Bound on $\Lambda$]

To improve inequality \eqref{boot correction}, first note that \(\left| \frac{\Pi E}{ H} - 1\right| = \frac{\Pi}{H} |\Lambda| \).
Then we may use the estimate of the correction in inequality \eqref{correction bound 3} to obtain
\begin{align}
			   \frac{\Pi }{ H } |\Lambda|
\lesssim		&	\frac{\Pi }{ H } \left[e^{-2\tau}|A| +  e^{-\tau}\left( 1 + e^{2 (\Pi E)^{1/2}} \right)\Pi E \right] \\
\lesssim	&	\frac{\Pi }{ H }\left[ e^{-2\tau} +  e^{-\tau}\Pi E \right] \\
\lesssim &  \epsilon^{1/2} e^{s_0/2}e^{3\tau/2} \left[  e^{-2\tau} +  e^{s_0-2\tau} \epsilon \right]\\
\lesssim &      \epsilon^{1/2} e^{s_0/2} e^{-\tau/2}+    \epsilon^{3/2} e^{3s_0/2}e^{-\tau/2}  \\
\lesssim &      \epsilon^{1/2} +    \epsilon^{3/2} e^{s_0} \\
\lesssim &      \epsilon^{1/2} 
\end{align}
since \(H^{-1} \lesssim e^{\tau}\).
Thus we may ensure 
\begin{align}
\left| \frac{\Pi E}{ H} - 1\right|  < 2
\end{align}
for \(\epsilon\) small.

\item[An Upper and Lower Bound on $H$]
For the energy $H$ we have the following estimate:
\begin{align}
\left| \partial_\tau \left( e^\tau H \right)  \right|
\lesssim e^\tau H \frac{\Pi E}{H} F+ \widetilde F 
\lesssim& e^\tau H  F+ \widetilde F
\end{align}
That is, there is a constant \(C >0\) such that
\begin{align}
\left| \partial_\tau \left( e^\tau H \right)  \right|
\leq& C  \left( e^\tau H  F+ \widetilde F\right).
\end{align}
The quantities \(F\) and \(\widetilde F\) are the nonnegative quantities defined in equations \eqref{f def} and \eqref{f tilde def}.
We then apply Lemma \ref{gronwall} to obtain the upper bound
\begin{align}
e^s H(s) \leq &\left(e^{s_0} H(s_0) +C  \int_{s_0}^s \widetilde F \, d\tau \right)\exp\left( C  \int_{s_0}^s F \, d\tau \right)\label{bound 1.1}
\end{align}
and the lower bound
\begin{align}
e^s H(s) \geq &2e^{s_0} H(s_0)-\left(e^{s_0} H(s_0) +C  \int_{s_0}^s \widetilde F \, d\tau \right)\exp\left(C  \int_{s_0}^s F \, d\tau\right).\label{bound 1.2}
\end{align}
What we want, then, is for \(\int_{s_0}^\infty\widetilde F \, d\tau\) to be bounded and for \( \int_{s_0}^\infty   F \, d\tau \to 0 \) as \( \epsilon \to 0\).

Recall that we have assumed \( e^{-\rho_0/2} < 1\) and \(|A| < 1\).
We compute
\begin{align}
 \widetilde F 
 =&  |A|\Pi\left(e^{-\tau}\left(1 +  \frac{\Pi_\tau}{\Pi}\right)  + e^{-\rho_0/2}  E^{1/2} \right) \\
 <&  e^{-\tau}\left(\Pi +  \Pi_\tau \right)  +  \Pi^{1/2} (\Pi E)^{1/2}  \\
\lesssim&  \epsilon^{1/2} e^{s_0/2}e^{-\tau/2}  +  \epsilon^{3/4} e^{3s_0/4}e^{-\tau/4} 
\end{align}
so
\begin{align}
\int_{s_0}^\infty \widetilde F \, d\tau  
\lesssim	&	\epsilon^{1/2} + \epsilon^{3/4} e^{s_0/2} \lesssim \epsilon^{1/4} . \label{f tilde bound}
\end{align}
Let \(C_1\) be the product of \(C \) and the constant associated to the \(\lesssim\) in inequality \eqref{f tilde bound}.

Inequality \eqref{bound 1.1} becomes
\begin{align}
e^s H(s) 
\leq &\left(e^{s_0} H(s_0) + C_1  \epsilon^{1/4} \right)\exp\left( C  \int_{s_0}^s F \, d\tau \right)
\end{align}
and the lower bound \eqref{bound 1.2} becomes
\begin{align}
e^s H(s) 
\geq &2e^{s_0} H(s_0)-\left(e^{s_0} H(s_0) +C_1  \epsilon^{1/4} \right)\exp\left(C  \int_{s_0}^s F \, d\tau\right)
\end{align}
Now we turn to the bound on \(F\).
\begin{align}
				F 
=			&	\int_{S^1} e^{\rho - \tau} \rho_\tau J \, d\theta +e^{-\tau} \left(1+e^{2(\Pi E)^{1/2}}\right)\left( \Pi + \Pi_\tau \right)  + (\Pi E)^{1/2} + e^{-\rho_0/2}e^{2(\Pi E)^{1/2}}  \Pi  E^{1/2} \\
\lesssim	&	e^{-3\tau}\int_{S^1} e^{\rho+l +2 \tau} J \, d\theta +e^{-\tau} \left( \Pi + \Pi_\tau \right)  + (\Pi E)^{1/2} +  \Pi  E^{1/2}\\
\lesssim	&	e^{-3\tau}Y_\tau + \epsilon^{1/2} e^{s_0/2}e^{-\tau/2} +\epsilon^{3/4} e^{3s_0/4}e^{-\tau/4}
\end{align}
We have previously bounded the integral of the latter terms in time by \(\epsilon^{1/4} \), so it remains to compute
\begin{align}
			\int_{s_0}^\infty e^{-3\tau} Y_\tau \, d\tau
\lesssim	\epsilon^{1/2} .
\end{align}
So
\begin{align}
\int_{s_0}^\infty F \, d\tau \lesssim \epsilon^{1/4} .
\end{align}
Thus, in total for \(H\), we have
\begin{align}
 e^s H(s)< \frac{3}{2}\left(e^{s_0} H(s_0) +C_1  \epsilon^{1/4}  \right)
\end{align}
when we choose \(\epsilon\) small enough that \(\exp\left( C  \int_{s_0}^s F \, d\tau \right) < \frac{3}{2}\).

Turning to the lower bound, it is useful to define \(N:=e^{s_0}H(s_0) \) and \( L:=C_1 \epsilon^{1/4}\).
Note assumption \eqref{eng lower bound} implies that
\begin{align}
\frac{1}{4(N+L)}\left(5N + 3L\right) >1 + \frac{1}{4M(N+L)} > 1 + \frac{1}{4M^2} 
\end{align}
so we take \(\epsilon\) small enough that
\begin{align}
1 + \frac{1}{4M^2}  > \exp\left( C  \int_{s_0}^s F \, d\tau \right).
\end{align}
The lower bound from Gr\"onwall's inequality takes the form
\begin{align}
e^s H(s) \geq &2N-\left(N +L \right)\exp\left( C  \int_{s_0}^s F \, d\tau \right) > 2N-\frac{1}{4}\left(5N + 3L\right)  = \frac{3}{4} (N - L)
\end{align}
which improves the lower bound on \(e^\tau H(\tau)\).

\item[Bounds on $\Pi, Y$]

Let us determine what the smallness assumptions of Lemma \ref{small lemma} imply for the error term of the ODE system of Section \ref{ode chapter}.
Recall the conclusion of Proposition \ref{cd bound}: if \(H> 0\), then
\begin{align}
\label{cd estimate}
\left|\left( \begin{array}{l} c \\ d \end{array} \right)\right|(s) \lesssim & e^{(s_0-s)/4}\left( \frac{ e^{s_0}H(s_0) }{e^s H(s)}\right)^{1/2}\left|\left( \begin{array}{l} c \\ d \end{array} \right)\right|(s_0) +\int^s_{s_0}     e^{(\tau-s)/4}\left( \frac{e^\tau H(\tau) }{e^s H(s)}\right)^{1/2}|\omega(\tau)| \, d\tau,
\end{align}
where 
\begin{align}
|\omega| \lesssim \left| \widetilde \Omega \right| = &\left| -\frac{1}{2} \partial_\tau \log\left( e^\tau H\right)\left( \begin{array}{l}\frac{2}{\sqrt {10}} \\ \frac{1}{\sqrt {10}} \end{array} \right)
+\left( \begin{array}{l} 0 \\ \frac{f(d,c)}{\left(c+\frac{2}{\sqrt{10}} \right)^2 }  +\left( \displaystyle\frac{\Pi E}{H}-1\right)\frac{d+\frac{1}{\sqrt{10}}}{\left(c+ \frac{2}{\sqrt{10}}\right)^2}+ \displaystyle\frac{\Omega}{e^{3\tau} H^{1/2}} \end{array} \right)  \right|
\end{align}
and
\begin{align}
\left|e^{-3\tau} H^{-1/2} \Omega
\right|
\lesssim& e^{-2\tau} |H|^{-1/2} E Y_\tau.
\end{align}

To begin with, note that \(e^\tau H(\tau)\) has both upper and lower bounds, and so both terms of the form \(\left( \frac{e^\tau H(\tau) }{e^s H(s)}\right)^{1/2}\) can be bounded above by a constant:
\begin{align}
\label{cd estimate2}
				\left|\left( \begin{array}{l} c \\ d \end{array} \right)\right|(s) 
\lesssim	&	e^{(s_0-s)/4}\left|\left( \begin{array}{l} c \\ d \end{array} \right)\right|(s_0) +\int^s_{s_0}     e^{(\tau-s)/4}|\omega(\tau)| \, d\tau \\
\lesssim	&	\epsilon + \int^s_{s_0}     e^{(\tau-s)/4}|\omega(\tau)| \, d\tau .
\end{align}
To finish the bootstrap, we must bound the right side of this inequality strictly below \(\epsilon^{1/4}\).
We deal with each of the 4 summands in \(\omega\) in the remainder of the proof.

The contribution to the right side of \eqref{cd estimate} from the error term \(e^{-3\tau} H^{-1/2} \Omega\) is
\begin{align}
					\int^s_{s_0}     e^{(\tau-s)/4}\left|e^{-3\tau} H^{-1/2} \Omega \right| \, d\tau 
\lesssim		&	e^{-s/4} \int^s_{s_0}     e^{-7\tau/4} \left|H^{-1/2}  \Pi^{-1}  \right|\Pi E Y_\tau\, d\tau \\
=				&	e^{-s/4} \int^s_{s_0}     e^{-11\tau/4} \left|\frac{\Pi E}{H}  \frac{1}{c + \frac{2}{\sqrt{10}}} \right|  Y_\tau\, d\tau \\
\lesssim		&	e^{-s/4} \int^s_{s_0}     e^{-11\tau/4} \left|\frac{\Pi E}{H}   Y_\tau \right| \, d\tau  \\
\lesssim		&	e^{-s/4} \int^s_{s_0}     e^{(-1/4 - 5/2)\tau}   Y_\tau  \, d\tau \\
\lesssim		&	\epsilon^{1/2} e^{(s_0-s)/4}       \\
<		&	\epsilon^{1/2}  
\end{align}
where we have used the fact that \(\frac{e^\tau H^{-1/2}}{c + \frac{2}{\sqrt{10}}}=	\Pi^{-1}\) and the bootstrap assumptions.

The contribution from \(\left( \frac{\Pi E}{H}-1\right)\frac{d+\frac{1}{\sqrt{10}}}{\left(c+ \frac{2}{\sqrt{10}}\right)^2}\) is 
\begin{align}
					\int^s_{s_0}     e^{(\tau-s)/4}\left|\left( \frac{\Pi E}{H}-1\right)\frac{d+\frac{1}{\sqrt{10}}}{\left(c+ \frac{2}{\sqrt{10}}\right)^2} \right| \, d\tau 
\lesssim		&	\int^s_{s_0}     e^{(\tau-s)/4}\epsilon^{1/2} \, d\tau  \\
\lesssim		&	\epsilon^{1/2} \left( 1 - e^{(s_0-s)/4} \right) \\
<		&	\epsilon^{1/2} .
\end{align}

Turning to \(\frac{f(d,c)}{\left(c+\frac{2}{\sqrt{10}} \right)^2 }\), we recall that \(f\) has vanishing linear part, so
\begin{align}
					\int^s_{s_0}     e^{(\tau-s)/4}\left|\frac{f(d,c)}{\left(c+\frac{2}{\sqrt{10}} \right)^2 }\right| \, d\tau 
\lesssim		&	\int^s_{s_0}     e^{(\tau-s)/4}\epsilon^{1/2} \, d\tau 
\lesssim	\epsilon^{1/2}  \left( 1- e^{(s_0-s)/4}\right)
<			\epsilon^{1/2}
\end{align}

To bound \(\partial_\tau \log\left( e^\tau H\right)\), note that \(e^\tau H\) has a lower bound, and use the estimates on \(F\) and \(\widetilde F\) obtained above to compute
\begin{align}
					\left|\partial_\tau \log\left( e^\tau H\right)\right|
=				&	\frac{1}{e^\tau H}\left| \partial_\tau (e^\tau H )		\right|\\
\lesssim		&	\frac{1}{e^\tau H}\left( e^\tau H F + \widetilde F	\right)\\
\lesssim		&  F + \widetilde F	\\
\lesssim		&	e^{-3\tau}Y_\tau + \epsilon^{1/2} e^{s_0/2}e^{-\tau/2} +\epsilon^{3/4} e^{3s_0/4}e^{-\tau/4} 
\end{align}
So the contribution to \eqref{cd estimate} is
\begin{align}
					\int^s_{s_0}     e^{(\tau-s)/4}\left( e^{-3\tau}Y_\tau + \epsilon^{1/2} e^{s_0/2}e^{-\tau/2} +\epsilon^{3/4} e^{3s_0/4}e^{-\tau/4} \right)  \, d\tau 
\lesssim		&	\epsilon^{1/2} + e^{-s/4}\int^s_{s_0}      e^{-\tau/4-5/2\tau}Y_\tau  \, d\tau \\
\lesssim		&	\epsilon^{1/2} + e^{(s_0-s)/4} \epsilon^{1/2} \\
\lesssim		&	\epsilon^{1/2}. 
\end{align}
Combining these estimates, we have from inequality \eqref{cd estimate2} that
\begin{align}
				\left|\left( \begin{array}{l} c \\ d \end{array} \right)\right|(s) 
\lesssim	&	e^{(s_0-s)/4}\left|\left( \begin{array}{l} c \\ d \end{array} \right)\right|(s_0) +\int^s_{s_0}     e^{(\tau-s)/4}|\omega(\tau)| \, d\tau \label{cd bound last}\\
\lesssim	&	\epsilon + \epsilon^{1/2} \\
\lesssim	&	 \epsilon^{1/2} .
\end{align}
This improves the bootstrap inequality on \(c,d\).
\end{description}
Thus we have improved all of the bootstrap inequalities, and the proof is complete.
\qedhere
\end{proof}

\section{Asymptotic Behavior}

\label{main chapter}

We are now in a position to present the \(B_0\) version of the main result of \cite{MR3513138}.
In particular, for \(T^2\)-symmetric vacuum spacetimes satisfying \(B=0\), we find rates of growth/decay in the expanding direction for the \(\theta\)-direction volume, the normalized energy, and their derivatives.
In going from the polarised to \(B_0\) case, we appear to lose some of the fine grained asymptotics of \(V\) and its mean.
Forthcoming work will describe the behavior of \(V\) and \(Q\), and the dependence of that behavior on the conserved quantity \(B\).
Given our estimates above, the proof of the theorem is nearly identical to the polarised case.

\begin{thm}\label{main theorem}
There exists an $\epsilon_0$ such that if $0 \leq \epsilon \leq \epsilon_0$, for any \(B_0\) initial data set satisfying the smallness conditions of Lemma \ref{small lemma}, the associated solution satisfies for \(\tau \in [s_0 , \infty) \).
\begin{align}
				\left|\left( \begin{array}{l} c \\ d \end{array} \right)\right|
\lesssim	&	e^{-\tau/4}																					\label{cd bound2}		\\
				\left|e^\tau H - C_\infty^2 \right| 
\lesssim	&	e^{-\tau/4}																					\label{tauh bound}		\\
				\left| e^{-\tau/2}\Pi -\frac{2}{\sqrt{10}} C_\infty \right|
\lesssim	&	e^{-\tau/4}																					\label{pi bound}		\\
				\left|e^{-5\tau/2}Y - \frac{1}{\sqrt{10}} C_\infty \right|
\lesssim	&	e^{-\tau/4}																					\label{y bound}			\\
				\left| e^{3\tau/2}  E -\frac{\sqrt{10}}{2 }C_\infty  \right| 
\lesssim	&	e^{-\tau/4}																					\label{e bound}			\\
				\left|\langle l \rangle - l\right|		
\lesssim	&	e^{-\tau/2}																					\label{l bound 1}		\\
				\left| e^l -\frac{1}{2}\right|
\lesssim	&	e^{-\tau/4}																					\label{l bound 2}		\\
				\left| \langle V\rangle - V \right|	+ \left|e^{V - \tau} \left( \langle Q \rangle - Q \right) \right|									\label{v bound 2}
\lesssim	&	e^{-\tau/2}										\\
				\left| \Pi^{-1}e^\rho-e^{\rho_\infty} \right|									\label{rho bound}
\lesssim	&	e^{-\tau/2}							
\end{align}
for some $C_\infty>0$ and $\rho_\infty \colon S^1 \to \mathbb R$. 
\end{thm}

\begin{proof}
The proof proceeds as in \cite{MR3513138}.
First, observe that inequalities \eqref{pi estimate} and \eqref{y estimate} imply that
\begin{align}
e^{-\tau} \Pi + e^{-\tau} \Pi_\tau + e^{-3\tau} Y \lesssim e^{-\tau/2}.
\end{align}
Furthermore, \(\Pi E \lesssim e^{-\tau}\) and \(e^\tau H\) is bounded above and below by positive constants.
On the other hand
\begin{align}
\left| \partial_\tau \left( e^\tau H \right)  \right|
\lesssim  e^\tau H  F+ \widetilde F
\lesssim    F+ \widetilde F
\lesssim   e^{-\tau/4} + e^{-3\tau} Y_\tau .
\end{align}
The right side is integrable in \(\tau\), so let \( C_\infty := \lim_{\tau \to \infty} \sqrt{ e^\tau H} \).
Then
\begin{align}
\left|C_\infty^2 - e^\tau H \right|\lesssim \int_\tau^\infty \left| \partial_\tau \left( e^s H \right)  \right| \, ds \lesssim e^{-\tau/4}
\end{align}
giving \eqref{tauh bound}.

Note that \eqref{cd estimate} now reads
\begin{align}
				\left|\left( \begin{array}{l} c \\ d \end{array} \right)\right|(s) 
\lesssim &	e^{-s/4} +\int^s_{s_0}     e^{(\tau-s)/4}|\omega(\tau)| \, d\tau,
\end{align}
and that all of the terms of \( \int^s_{s_0}     e^{(\tau-s)/4}|\omega(\tau)| \, d\tau\) are now bounded by \(e^{-s/4}\) with the exception of
\begin{align}
\int^s_{s_0}     e^{(\tau-s)/4}\left|\frac{f(d,c)}{\left(c+\frac{2}{\sqrt{10}} \right)^2 }\right| \, d\tau  \lesssim \int^s_{s_0}     e^{(\tau-s)/4}	\left|\left( \begin{array}{l} c \\ d \end{array} \right)\right|^2 \, d\tau \lesssim \int^s_{s_0}     e^{(\tau-s)/4}	\left|\left( \begin{array}{l} c \\ d \end{array} \right)\right| \, d\tau  \label{f bound2}
\end{align}
since \(\left|\left(c,d \right)\right| \lesssim 1\).
So
\begin{align}
				\left|\left( \begin{array}{l} c \\ d \end{array} \right)\right|(s) 
\lesssim &	e^{-s/4} +\int^s_{s_0}     e^{(\tau-s)/4}|\omega(\tau)| \, d\tau \\
\lesssim &	e^{-s/4} +\int^s_{s_0}     e^{(\tau-s)/4}\left|\left( \begin{array}{l} c \\ d \end{array} \right)\right| \, d\tau .
\end{align}
Applying the integral version of Gr\"onwall's inequality gives
\begin{align}
				\left|\left( \begin{array}{l} c \\ d \end{array} \right)\right|(s) 
\lesssim &	e^{-s/4} + \int^s_{s_0}e^{-\tau/4} e^{(\tau-s)/4} \exp \left(\int^s_{\tau}     e^{(r-s)/4} \, dr \right)\, d\tau \\
\lesssim &	e^{-s/4} + e^{-s/4}\int^s_{s_0}  \exp \left(\int^s_{\tau}     e^{(r-s)/4} \, dr \right)\, d\tau  \\
\lesssim &	e^{-s/4} + se^{-s/4} \\
\lesssim &	e^{\left(\delta -\frac{1}{4}\right)s} 
\end{align}
for any \(\delta > 0\).
Inserting this improved estimate into \eqref{f bound2} and applying Gr\"onwall's inequality again gives \eqref{cd bound2}.
Combining this with \eqref{tauh bound} yields \eqref{pi bound} and \eqref{y bound}.

Recall that \(H = \Pi (E + \Lambda)\) and
\begin{align}
\left | \Lambda \right| \lesssim e^{-2\tau}|A| + e^{-\tau} \left( 1 + e^{2(\Pi E)^{1/2} } \right) \Pi E \lesssim e^{-2\tau}.
\end{align}
Then combine \eqref{tauh bound} and \eqref{pi bound} to obtain \eqref{e bound}.
The estimate \eqref{l bound 1} follows from \eqref{l theta bound} and \eqref{e bound}.
Estimate \eqref{v bound 2} follows directly from the Poincar\'e inequality and the bound on \(E\).

To estimate \(e^l\) let us note that
\begin{align}
				\left| \Pi  e^l  - e^{-2\tau}Y \right| 
=			&	\left| \int_{S^1} e^{\rho(\varphi)+l(\varphi)} \left( e^{l(\varphi)-l(\theta)} - 1 \right)   \, d\varphi\right|
\lesssim		e^{-2\tau} Y. \label{el bound 1}
\end{align}
One can then combine \eqref{el bound 1}, \eqref{pi bound}, and \eqref{y bound} to find that \(\sup_{S^1} e^l\) is bounded by a constant.
Then, we estimate again
\begin{align}
				\left| \Pi  e^l  - e^{-2\tau}Y \right| 
=			&	\left| \int_{S^1} e^{\rho(\varphi)} \left( e^{l(\varphi)} - e^{l(\theta)} \right)   \, d\varphi\right|
\lesssim		\Pi \left(\sup_{S^1}e^l\right) e^\tau E
\lesssim 1 \label{el bound 2}
\end{align}
and combine \eqref{el bound 2}, \eqref{pi bound}, and \eqref{y bound} to obtain \eqref{l bound 2}.
The proof of \eqref{rho bound} is identical to the one in \cite{MR3513138}, since we have the same bounds on \(\rho_{\theta \tau}\).
\end{proof}

\section{Numerical Evidence}

\label{numerical chapter}

The full Einstein Flow is a large, quasilinear system of partial differential equations about which it is difficult even to make conjectures.
This remains true even in the simplified \(T^2\)-symmetric case considered in this work.
It has been crucial to this work to base our conjectures on evidence garnered from numerical simulations of \(T^2\)-symmetric Einstein Flows.
We summarize this numerical work in this section.
A more detailed discussion of the numerical methods and results is the subject of a forthcoming paper.

Our code is a reimplementation of one previously developed by Berger to simulate \(T^2\)-symmetric spacetimes in the contracting direction \cite{MR1858721}, and then later in the expanding direction.
We reimplemented this code in OCaml\footnote{OCaml is a general purpose programming language developed primarily at INRIA. See \url{https://ocaml.org/}.}, and made a number of modifications to improve the accuracy and speed.
Most importantly, we developed code to produce solutions of the \(T^2\)-symmetric constraint equation via a random process, which allowed us to probe the behavior of generic \(T^2\)-symmetric Einstein Flows.

We have developed code which samples the constraint submanifold for the \(T^2\)-symmetric Einstein Field Equations in a fairly generic manner.
We have then evolved these initial data using a finite difference method.
This generic sampling has been a crucial element allowing us to determine that the assumption \(B=0\) was necessary for our main theorem, and otherwise develop our intuition about the solutions.
The simulations have the expected convergence properties upon refining the spatial resolution so we are confident that they are accurate approximations of solutions.
To obtain confidence that our simulations depict behavior which is generic for the class under consideration, we simulated on the order of 20 randomly chosen initial constraints solutions in each of the following classes: polarised, \(B_0\), and \(B\ne0\) \(T^2\)-symmetric.
The qualitative behavior depicted in Figures \ref{numerical plots rl} through \ref{numerical plots v} is observed to be the same for all simulations in that class.

It has been useful to plot the evolution of the following quantities along each of the numerical solutions.
\begin{align}
S := \partial_\tau \int_{S^1}  l \, e^{\rho - \tau / 2}d\theta , \quad 
&T :=  \partial_\tau \int_{S^1}  \rho \, e^{\rho - \tau / 2}d\theta , \\ 
E_V := \int_{S^1} \left[V_\tau^2 + e^{2(\tau-\rho)} V_\theta^2\right] \, e^{\rho - \tau / 2}d\theta , \quad 
&E_Q := \int_{S^1}e^{2(V-\tau)} \left( Q_\tau^2 + e^{2(\tau-\rho)}  Q_\theta^2 \right) \, e^{\rho - \tau / 2}d\theta ,\\
 W:= \log \left| \int_{S^1}  V_\tau \, e^{\rho - \tau / 2}d\theta\right|
\end{align}
These are not the quantities that were used in the proof of our main theorem, but they capture the dynamics of the system.
The volume form \(e^{\rho - \tau / 2}d\theta\) is used to smooth out the graphs (the integrals generally oscillate without this normalization).

\begin{figure}[H]
   \centering
   \subfloat[polarised]{
       \includegraphics[width=0.3\textwidth]{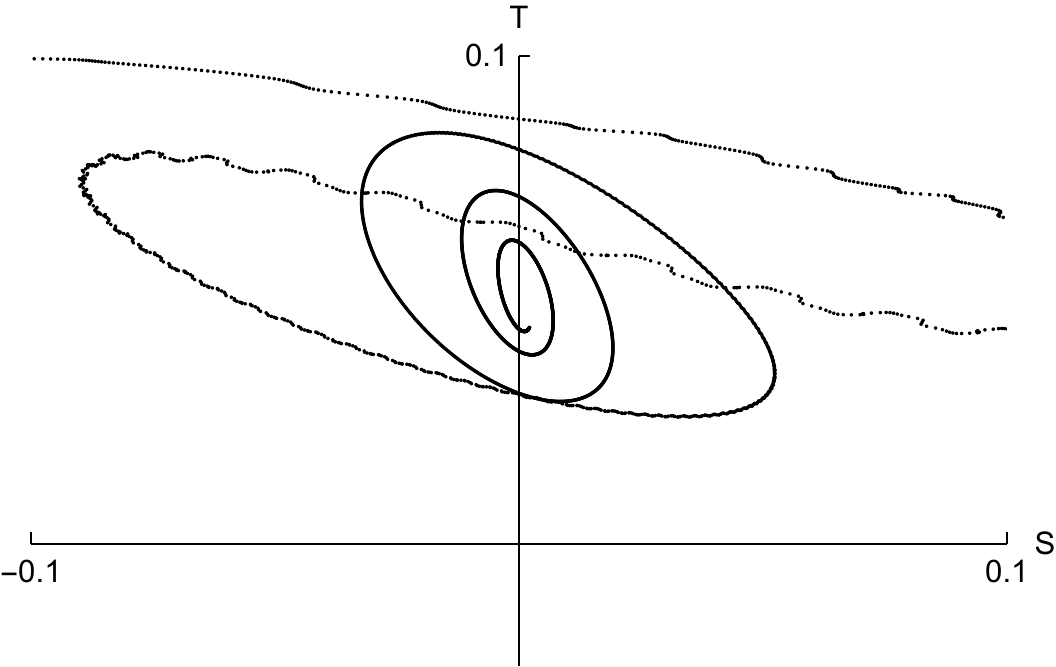}
   }
   \subfloat[\(B=0\)]{
       \includegraphics[width=0.3\textwidth]{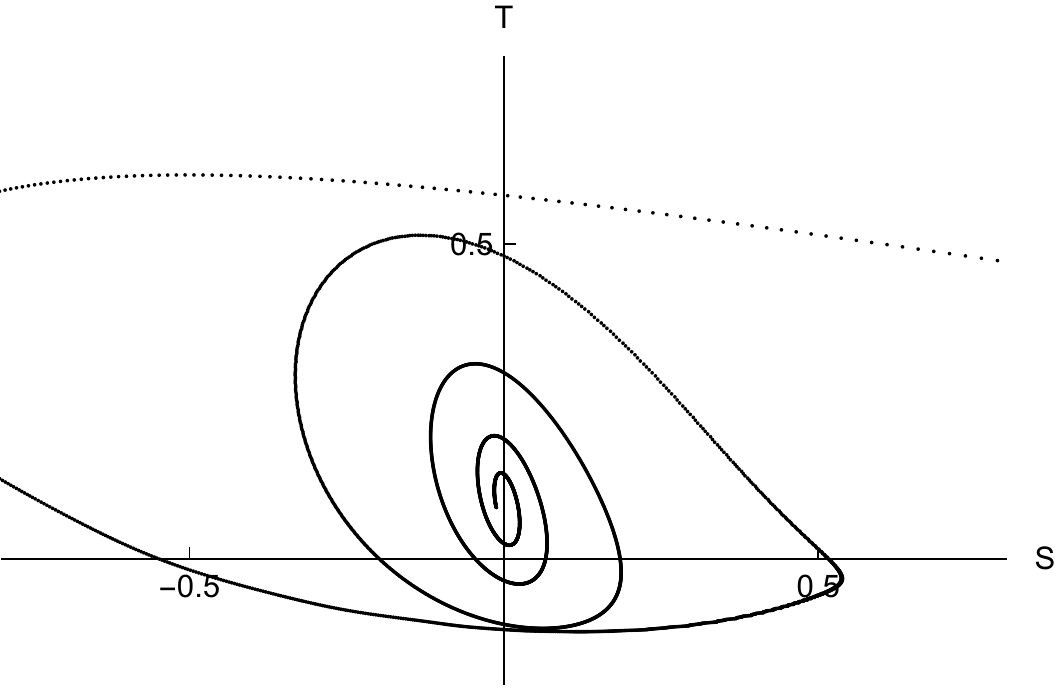}
   }
   \subfloat[\(B \ne 0\)]{
       \includegraphics[width=0.3\textwidth]{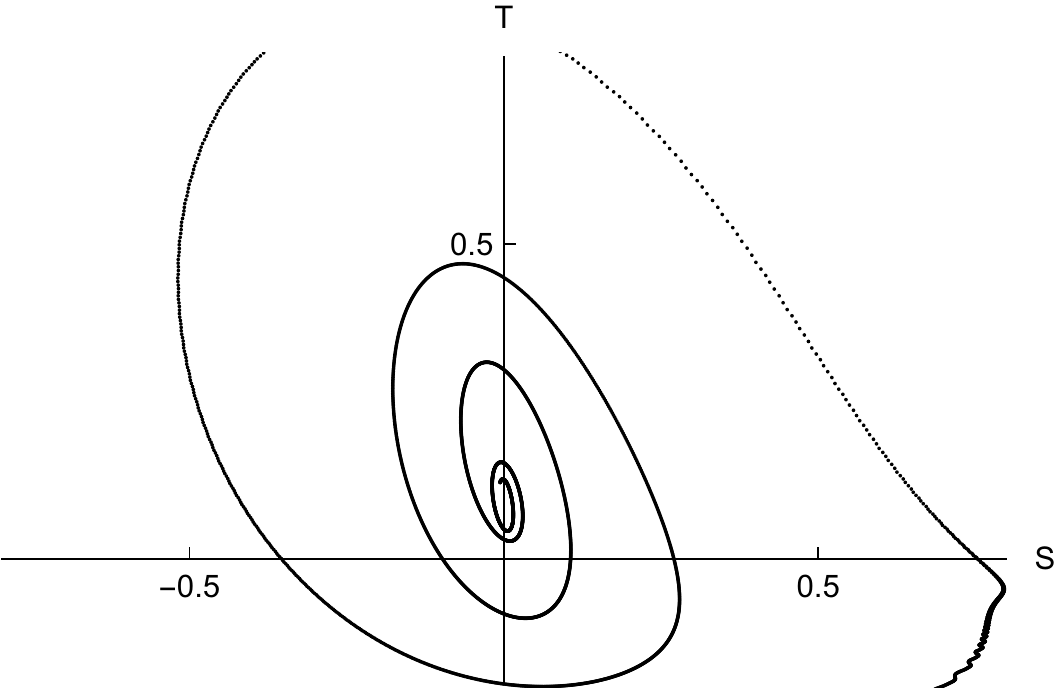}
   }
\caption{\(S\) and \(T\) flow toward a spiral sink, regardless of polarisation or the value of \(B\). Although \(l_\tau, \rho_\tau\) converge to the same values in all cases, the volume form \( e^{\rho - \tau/2} \, d\theta \) causes the variables used in the plots to flow toward different values.}\label{numerical plots rl}
\end{figure}
\begin{figure}[H]

   \centering
   \subfloat[polarised]{
       \includegraphics[width=0.3\textwidth]{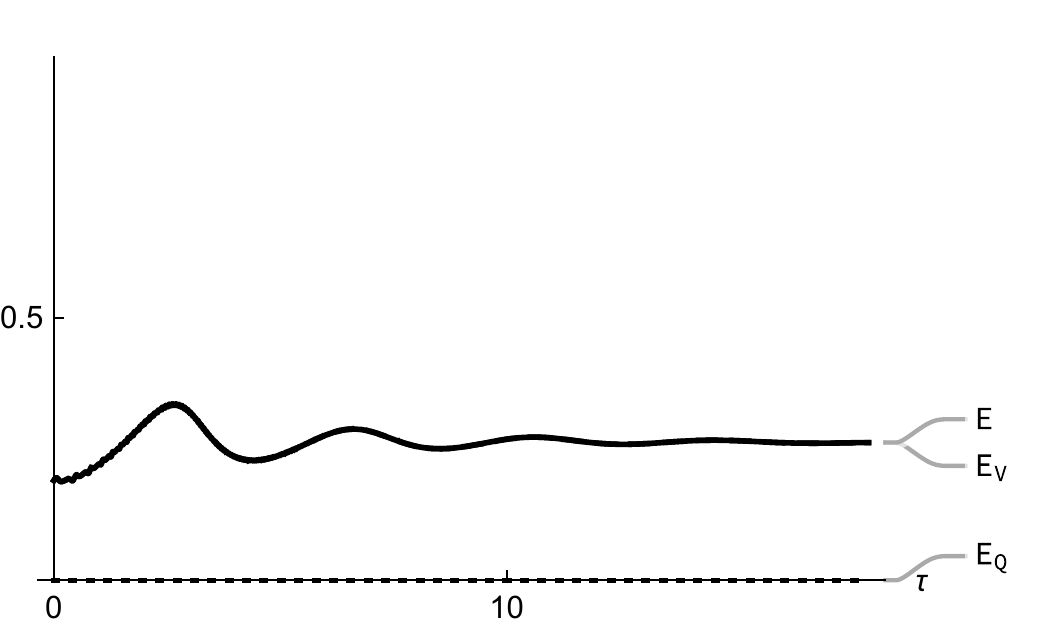}
   }
   \subfloat[\(B=0\)]{
       \includegraphics[width=0.3\textwidth]{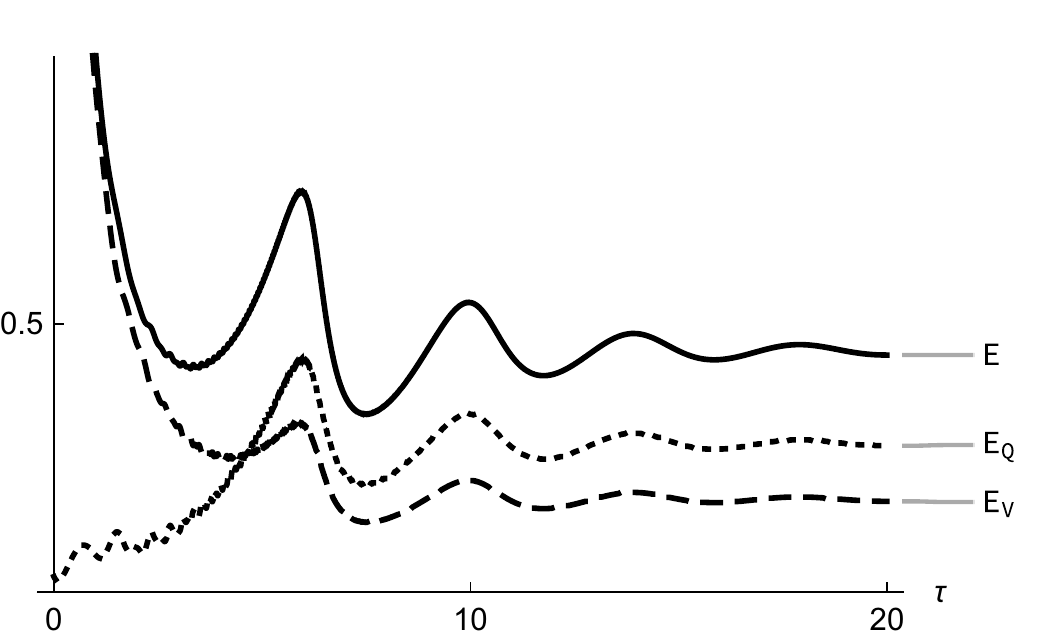}
   }
   \subfloat[\(B \ne 0\)]{
       \includegraphics[width=0.3\textwidth]{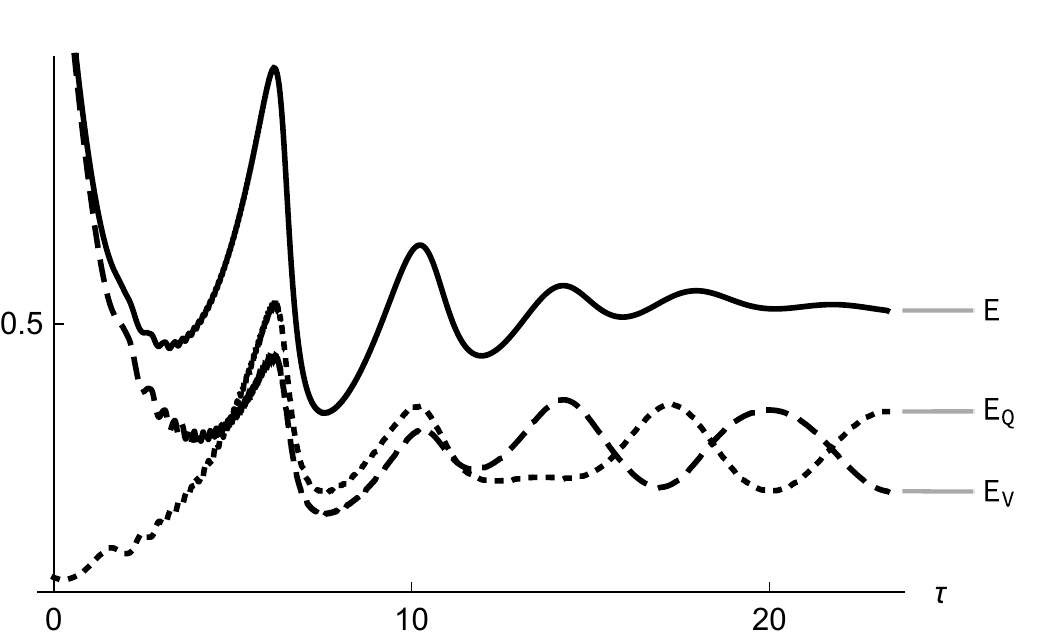}
   }
\caption{For polarised solutions, \(E = E_V\) which converges to a constant. For \(B=0\) solutions, \(E\) and the \(V\) and \(Q\) energies all converge to constants.
For \(B\ne0\) solutions, however, although the total energy converges, \(E_V\) and \(E_Q\) do not; they oscillate with amplitude which does not decay and period matching the period of the sink in Figure \ref{numerical plots rl}.}\label{numerical plots e}
\end{figure}
\begin{figure}[H]

   \centering
   \subfloat[polarised]{
       \includegraphics[width=0.3\textwidth]{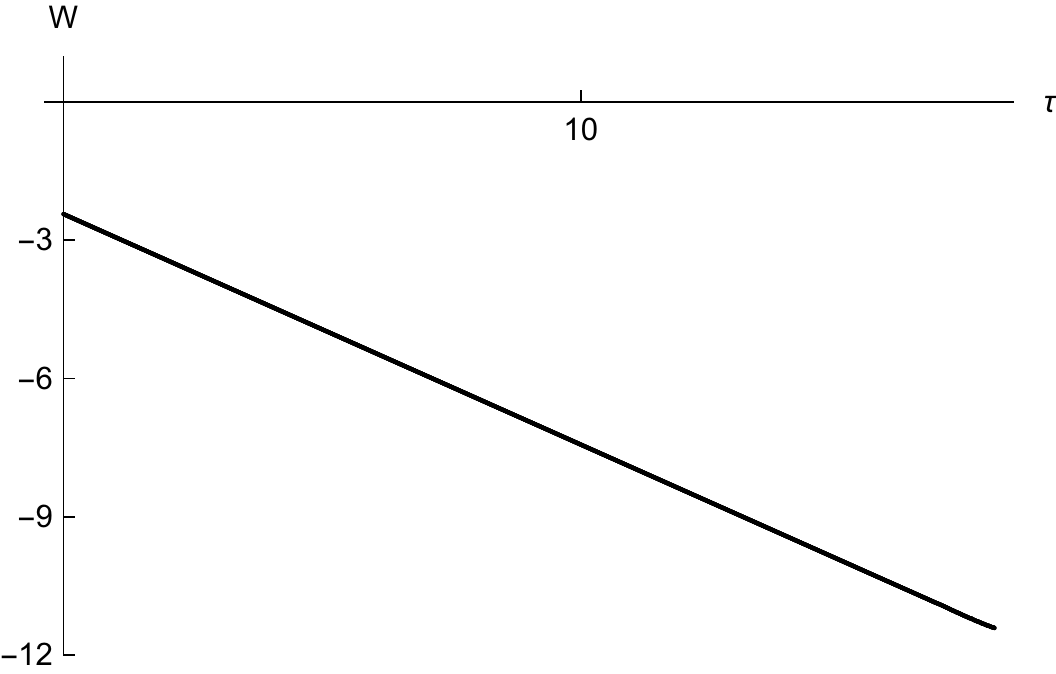}
   }
   \subfloat[\(B=0\)]{
       \includegraphics[width=0.3\textwidth]{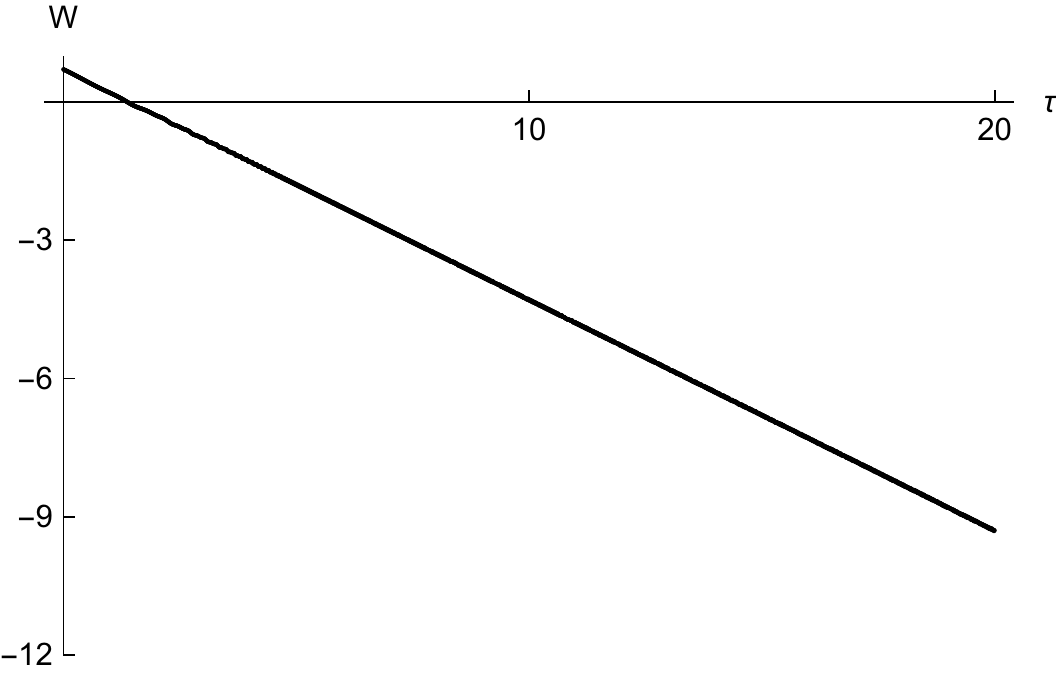}
   }
   \subfloat[\(B \ne 0\)]{
       \includegraphics[width=0.3\textwidth]{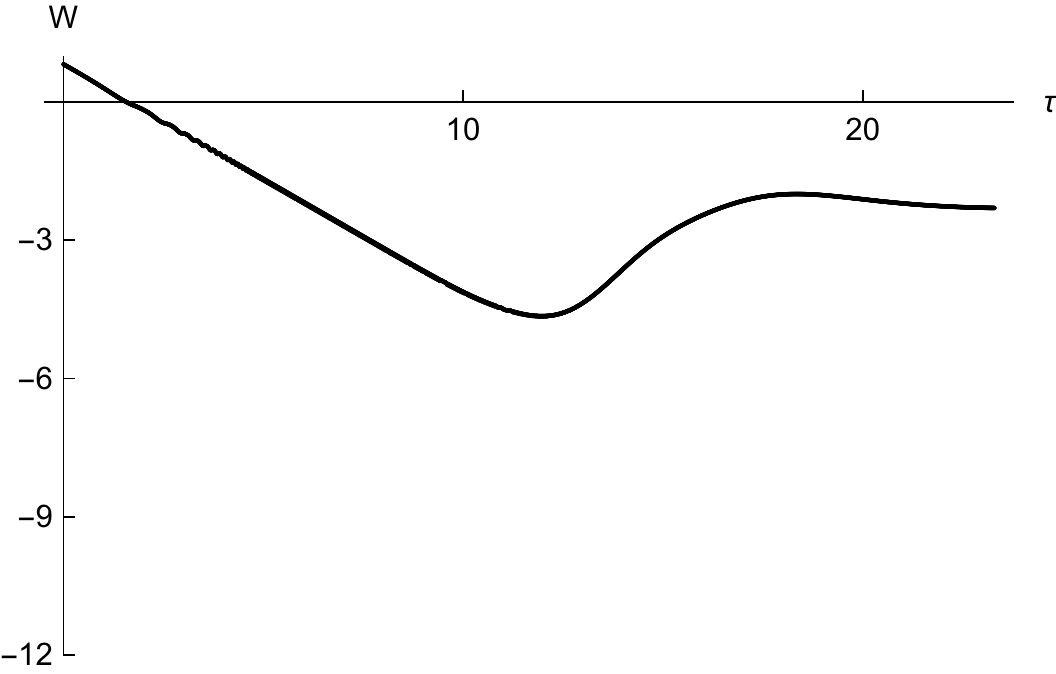}
   }
\caption{For \(B=0\) solutions (including polarised), \(V_\tau \to 0\) exponentially.
For \(B \ne 0\) solutions, however, \(V_\tau\) appears to converge to a nonzero constant.}\label{numerical plots v}
\end{figure}

In \cite{MR3513138}, the authors are able to determine the first order behavior of the energy and \(\Pi\), but also the first order behavior of \(V\) and the rate of its decay to the mean value.
We have generalized their results on the asymptotic values of the energy, \(\Pi\) as well as the decay of \(V\) and \(Q\) to their means to the \(B_0\) case, but so far have been unable to derive other estimates for \(V\) and \(Q\).
However, the numerical solutions that we have found have the property that there are constants \(a,C_V\) such that
\begin{align}
 \left| \langle V \rangle  - C_V \tau - a \right| = O(e^{-\tau/2})
\end{align}
and that
\begin{align}
C_V = \left \{ \begin{array}{ll} 0 & \text{if } B = 0 \\ \frac{1}{2} & \text{if } B \ne 0 \end{array}\right..
\end{align}
More detailed descriptions of the numerical results will be given in future work.

\appendix
\section{Concordance of notations between \cite{MR1474313}, \cite{MR1858721}, \cite{MR3312437}, and \cite{MR3513138}}

The Einstein Flows under consideration in the this work have been studied extensively, including many important special subsets of solutions.
Unfortunately, authors have used many different coordinates for exactly the same set of spacetimes, and this document adds yet another set of coordinates.
As an aid to the reader who wishes to read the cited works together, we provide in this appendix a concordance of notations used in the most frequently cited of these works.

To the best of our knowledge, all of the works in the table rely on the foliation and equations derived in \cite{MR1474313}.
This paper, \cite{MR1474313}, \cite{MR1858721} and \cite{MR3312437} use coordinates for \(T^2\)-symmetric Einstein Flows which are completely general.
The analysis in \cite{MR3513138} applies only to polarised \(T^2\)-symmetric Einstein Flows, and so relies on the assumption that some metric components vanish identically.
In \cite{MR2032917}, future asymptotics of Gowdy solutions are derived.
The notation used there is exactly the notation of \cite{MR3312437} if one imposes the conditions \(\alpha \equiv 1, K=0\) so we omit it from the table.

In the table below, each column uses the notation internal to the document named in the first row.
All of the expressions in a given row are equal.
For example, the function called \(P\) in \cite{MR3312437} has the expression \(2 U - \log R\) in \cite{MR3513138}.
Since \cite{MR3513138} only deals with polarised flows, the expressions in this column are only equal to those in other documents if the polarization condition is imposed.

\newpage
\phantom{A}
\vfill
\begin{center}
\begin{adjustbox}{angle=90}
{\renewcommand{\arraystretch}{3}
\begin{tabular}{p{0.16\textheight}p{0.16\textheight}p{0.16\textheight}p{0.10\textheight}p{0.22\textheight}}
	\cite{MR1474313}								&	\cite{MR1858721}																&	\cite{MR3312437}																							&	\cite{MR3513138}																&	this document					\\\hline
	$\log t$											&	\(-\tau\)																			&	$\log t$																										&	$\log R$																			&	$\tau$							\\\hdashline[0.5pt/5pt]
	$2U$												&	$P-\tau$																			&	$P+\log t$																									&	$2 U$																				&	$V$								\\\hdashline[0.5pt/5pt]
	$\alpha^{-1/2}$									&	\(2 \pi_\lambda\)																&	$\alpha^{-1/2}$																								&	$a^{-1}$																			&	$e^\rho$							\\\hdashline[0.5pt/5pt]
	$\displaystyle\frac{K^2}{2}t^{-2}\alpha e^{2\nu}$		&	$\displaystyle\frac{K^2}{2}e^{P + \frac{1}{2} \lambda + 3\tau / 2}  $			&	$\displaystyle\frac{K^2}{2}t^{-3/2}e^{P + \frac{1}{2} \lambda}  $											&	$\displaystyle\frac{K^2}{2}R^{-2}e^{2\eta}$												&	$e^l$								\\\hdashline[0.5pt/5pt]
	$2t U_t$											&	$1- \displaystyle\frac{\pi_P}{2\pi_\lambda}$							&	$tP_t + 1  $																									&	$2R U_R$																			&	$V_\tau$								\\\hdashline[0.5pt/5pt]
	$\displaystyle t^{-1}\int_{S^1} \alpha^{-1/2} e^{4U} Q_t \, d\theta$												&	$\displaystyle -\int_{S^1} \pi_Q  \, d\theta$											&	$\displaystyle \int_{S^1} \alpha^{-1/2} e^{2P}  tQ_t \, d\theta$												&	$0$																					&	$\displaystyle \int_{S^1}  e^{\rho+2(V-\tau)} Q_\tau  \, d\theta=: B$								\\\hdashline[0.5pt/5pt]
\end{tabular}
}
\end{adjustbox}
\end{center}
\vfill
\phantom{A}

\bibliographystyle{plain}	
\bibliography{bib}

\def\cprime{$'$} \def\cprime{$'$}
\begin{thebibliography}{10}

\bibitem{MR3085923}
Ellery Ames, Florian Beyer, James Isenberg, and Philippe~G. LeFloch.
\newblock Quasilinear hyperbolic {F}uchsian systems and {AVTD} behavior in
  {$T^2$}-symmetric vacuum spacetimes.
\newblock {\em Ann. Henri Poincar\'{e}}, 14(6):1445--1523, 2013.

\bibitem{MR1888088}
Michael~T. Anderson.
\newblock On long-time evolution in general relativity and geometrization of
  3-manifolds.
\newblock {\em Comm. Math. Phys.}, 222(3):533--567, 2001.

\bibitem{MR0363384}
V.~A. Belinski\u{\i} and I.~M. Khalatnikov.
\newblock Effect of scalar and vector fields on the nature of the cosmological
  singularity.
\newblock {\em \v{Z}. \`Eksper. Teoret. Fiz.}, 63:1121--1134, 1972.

\bibitem{bergerpcgm}
Beverly~K. Berger.
\newblock Comments on expanding {$T^2$} symmetric cosmological spacetimes.
\newblock 31st Pacific Coast Gravity Meeting, 2015.

\bibitem{bergeraps}
Beverly~K. Berger.
\newblock Transitions in expanding cosmological spacetimes.
\newblock APS April Meeting 2015, 2015.

\bibitem{MR1474313}
Beverly~K. Berger, Piotr~T. Chru{\'s}ciel, James Isenberg, and Vincent
  Moncrief.
\newblock Global foliations of vacuum spacetimes with {$T^2$} isometry.
\newblock {\em Ann. Physics}, 260(1):117--148, 1997.

\bibitem{MR1858721}
Beverly~K. Berger, James Isenberg, and Marsha Weaver.
\newblock Oscillatory approach to the singularity in vacuum spacetimes with
  {$T^2$} isometry.
\newblock {\em Phys. Rev. D (3)}, 64(8):084006, 20, 2001.

\bibitem{clausenthesis}
A.~{Clausen}.
\newblock {\em {Singular behavior in {$T^2$} symmetric spacetimes with
  cosmological constant}}.
\newblock PhD thesis, University of Oregon, 2007.

\bibitem{MR1894914}
Arthur~E. Fischer and Vincent Moncrief.
\newblock The reduced {E}instein equations and the conformal volume collapse of
  3-manifolds.
\newblock {\em Classical Quantum Gravity}, 18(21):4493--4515, 2001.

\bibitem{MR0339764}
Robert~H. Gowdy.
\newblock Vacuum spacetimes with two-parameter spacelike isometry groups and
  compact invariant hypersurfaces: topologies and boundary conditions.
\newblock {\em Ann. Physics}, 83:203--241, 1974.

\bibitem{MR1939922}
James Isenberg and Vincent Moncrief.
\newblock Asymptotic behaviour in polarized and half-polarized {U{$(1)$}}
  symmetric vacuum spacetimes.
\newblock {\em Classical Quantum Gravity}, 19(21):5361--5386, 2002.

\bibitem{MR1501305}
Edward Kasner.
\newblock Solutions of the {E}instein equations involving functions of only one
  variable.
\newblock {\em Trans. Amer. Math. Soc.}, 27(2):155--162, 1925.

\bibitem{MR3513138}
Philippe LeFloch and Jacques Smulevici.
\newblock Future asymptotics and geodesic completeness of polarized
  {T}2-symmetric spacetimes.
\newblock {\em Anal. PDE}, 9(2):363--395, 2016.

\bibitem{2017arXiv170105150L}
J.~{Lott}.
\newblock {Collapsing in the Einstein flow}.
\newblock {\em ArXiv e-prints}, January 2017.

\bibitem{2018CQGra..35c5010L}
J.~{Lott}.
\newblock {Backreaction in the future behavior of an expanding vacuum
  spacetime}.
\newblock {\em Classical Quantum Gravity}, 35(3):035010, 10, 2018.

\bibitem{Radermacher:2017mkr}
Katharina Radermacher.
\newblock {On the Cosmic No-Hair Conjecture in T2-symmetric non-linear scalar
  field spacetimes}.
\newblock 2017.

\bibitem{MR2032917}
Hans Ringstr{\"o}m.
\newblock On a wave map equation arising in general relativity.
\newblock {\em Comm. Pure Appl. Math.}, 57(5):657--703, 2004.

\bibitem{MR3312437}
Hans Ringstr{\"o}m.
\newblock Instability of {S}patially {H}omogeneous {S}olutions in the {C}lass
  of {$\Bbb{T}^2$}-{S}ymmetric {S}olutions to {E}instein's {V}acuum
  {E}quations.
\newblock {\em Comm. Math. Phys.}, 334(3):1299--1375, 2015.

\bibitem{2017arXiv170702803R}
Hans Ringstr{\"o}m.
\newblock {Linear systems of wave equations on cosmological backgrounds with
  convergent asymptotics}.
\newblock {\em ArXiv e-prints}, July 2017.

\bibitem{2014arXiv1407.6298R}
Igor Rodnianski and Jared Speck.
\newblock {Stable Big Bang Formation in Near-FLRW Solutions to the
  Einstein-Scalar Field and Einstein-Stiff Fluid Systems}.
\newblock {\em ArXiv e-prints}, July 2014.

\bibitem{MR3739229}
Igor Rodnianski and Jared Speck.
\newblock A regime of linear stability for the {E}instein-scalar field system
  with applications to nonlinear big bang formation.
\newblock {\em Ann. of Math. (2)}, 187(1):65--156, 2018.

\bibitem{MR3874696}
Igor Rodnianski and Jared Speck.
\newblock Stable {B}ig {B}ang formation in near-{FLRW} solutions to the
  {E}instein-scalar field and {E}instein-stiff fluid systems.
\newblock {\em Selecta Math. (N.S.)}, 24(5):4293--4459, 2018.

\bibitem{MR2960029}
Jared Speck.
\newblock The nonlinear future stability of the {FLRW} family of solutions to
  the {E}uler-{E}instein system with a positive cosmological constant.
\newblock {\em Selecta Math. (N.S.)}, 18(3):633--715, 2012.

\end{thebibliography}

\end{document}